\documentclass[11pt]{amsart}
\usepackage{amsmath, amsthm, amscd, amssymb, amsfonts, amsxtra, verbatim}
\usepackage{enumerate}
\usepackage{inputenc, t1enc}
\usepackage{float}
\usepackage{tgpagella}

\theoremstyle{plain}
\newtheorem{lema}{Lemma}[section]
\newtheorem{teo}[lema]{Theorem}
\newtheorem{coro}[lema]{Corollary}

\newtheorem{prop}[lema]{Proposition}

\theoremstyle{definition}

\newtheorem{ejem}[lema]{Example}

\newtheorem{rem}[lema]{Remark}

\theoremstyle{remark}
\newtheorem*{note}{Note}

\numberwithin{equation}{section}

\newcommand{\N}{\mathbb N}
\newcommand{\Z}{\mathbb Z}
\newcommand{\Q}{\mathbb Q}
\newcommand{\R}{\mathbb R}

\newcommand{\C}{\mathbb C}

\newcommand{\f}{\frac}
\newcommand{\tf}{\tfrac}

\newcommand{\sk}{\smallskip}
\newcommand{\msk}{\medskip}
\newcommand{\bsk}{\bigskip}

\usepackage{color}

\title[Krawtchouk polynomials, binomial coefficients and Catalan numbers]{New identities for binary Krawtchouk polynomials, 
binomial coefficients and Catalan numbers}
\address{Ricardo A.\@ Podest\'a -- Centro de Investigaci\'on y Estudios en Matem\'atica de C\'ordoba (UNC - CONICET), FaMAF, (5000)
 C\'ordoba, Argentina.} 
\email{podesta@famaf.unc.edu.ar}
\author{Ricardo A.\@ Podest\'a}
\thanks{Partially supported by CONICET, FONCyT and SECyT-UNC}

\begin{document}
\bibliographystyle{plain}

\begin{abstract}
We obtain new combinatorial identities for integral values of binary \linebreak Krawtchouk polynomials $K^{2m}_p(x)$, $0\le p\le 2m$, 
by computing the characters of the $p$-exterior representations on certain elements of order 2 of $\mathrm{SO}(2m)$. 
From this identities, we deduce several new relations for binomial coefficients and Catalan numbers. 
\end{abstract}

\maketitle

\section{Introduction}
For each $0\le p\le n$, the \textit{binary Krawtchouk polynomial} (BKP for short) of order $n$ and degree $p$ is defined by 
\begin{equation} \label{eq. kraws}
K_p^n(x) = \sum_{j=0}^p (-1)^j \tbinom xj \tbinom{n-x}{p-j} 
\end{equation}
where $\tbinom xj = x(x-1)\cdots (x-j+1)/j!$ for $j\ge 1$ and $\tbinom x0 =1$. 
Thus, by definition, $K_p^n(j) \in\Z$ for every integer $j$. 
One can easily check that  
\begin{equation} \label{k0k1}
K_p^n(0)=\tbinom np, \qquad K_p^n(1) = (1-\tfrac{2p}{n}) \tbinom np, \qquad K_p^n(n) = (-1)^p \tbinom np
\end{equation}
and that we also have $K_0^n(x)=1$ and $K_1^n(x) = n-2x$.

These polynomials form a discrete family $\{K_p^n(x)\}_{p=0}^n$ of orthogonal polynomials with respect to the binomial distribution. 
They satisfy several identities such as orthogonality, $3$-term recursions in the $3$ variables, modularity properties, integral formulae, relations with other families of orthogonal polynomials, etc. See \cite{KL2} for a survey on binary Krawtchouk polynomials and its properties (also \cite{ChS} and \cite{KL}).

\msk
Binary Krawtchouk polynomials appear in several problems related with the abelian group $\Z_2^k$, for some $k$. The most commonly known examples of this are applications to combinatorial problems or to coding theory (see \cite{KL} for a survey). In combinatorics, BKP's appear in: 
(a) the existence or not of the inverse of the Radon transform on $\Z_2^k$ (\cite{DG}), 
(b) reconstruction problems on graphs (switching, reorientation, sign \cite{St}) and (c) multiple perfect coverings of $\Z_2^n$ 
(\cite{HHKL}, \cite{WCL}). Also, in the context of binary codes, BKP's play a role in: (a) the existence or not of binary perfect codes 
(\cite{vL}), (b) alternative expressions for the MacWilliams identities relating the weight enumerator of the code with the corresponding enumerator of its dual (\cite{HP}) and (c) in some universal bounds for codes 
(\cite{Le}). 
Notably, in all of these problems, the relevant question is the existence or not of integral zeros of the BKP's involved 
(see \cite{ChS}, \cite{Ha}, \cite{Ha2}, \cite{HS}, \cite{KL} 
for results related to integral zeros of BKP's).

Less known is the ubiquity of Krawtchouk polynomials in spectral geometry.
In this setting, BKP's were used to study isospectrality problems for elliptic differential operators $D$
acting on $\Z_2^k$-manifolds (i.e.\@ compact flat Riemannian manifolds having holonomy group isomorphic to $\Z_2^k$). 
Here, $D$ is either a Dirac-type operator (spin Dirac or signature operator, see \cite{MP}, \cite{MP2}, \cite{MP3}, \cite{MPR}) or a Laplacian (Hodge Laplacian, $p$-Laplacian, full Laplacian, see \cite{MPR}, \cite{MR3}, \cite{MR2}, \cite{MR1}). 
Again, the existence of integral zeros play a key role in the results.

\sk
%\subsubsection*{Outline and results}
In brief, in the first two sections after the Introduction, we obtain new identities for integral values of binary Krawtchouk polynomials and in the subsequent sections we give applications of them %se identities 
to binomial coefficients and Catalan numbers. 

An outline of the paper is as follows. In Section 2, we first recall some facts on the $p$-exterior representations of 
$\mathrm{SO}(2m)$, $0\le p\le 2m$, and give an explicit expression for $\chi_p(x)$, the character values of these representations at elements $x$ of the maximal torus $T_{2m}$ of $\mathrm{SO}(2m)$. 
Then, by relating the $p$-characters $\chi_p(x)$ at elements $x$ of order 2 with BKP's, we obtain a new identity for binary Krawtchouk polynomials (see Theorem~\ref{teo1}) of even order at even values and degree $p$, i.e.\@ $K_p^{2m}(2j)$, 
in terms of the integral values $K_k^m(j)$ of polynomials of smaller degree $k$ (see also Corollary~\ref{coro alt}). 

In Section 3, we generalize the reduction formula obtained in Theorem \ref{teo1} (see Theorem~\ref{teo2}) giving rise to a whole new family of identities for BKP's of the form $K_p^{2^r m}(2^s j)$. We then exhibit several explicit computations illustrating 
the results.

In Section 4, by evaluating the expressions previously obtained for BKP's, we present recursive relations between binomial coefficients (see \eqref{comb3}--\eqref{comb1}). In particular, expressions for $\tbinom{2m}{2q}$, $\tbinom{2m}{2q+1}$, $\tbinom{2m+1}{2q}$ and $\tbinom{2m+1}{2q+1}$ in terms of $\tbinom mq$ and falling factorials $(q)_0,(q)_1,\ldots,(q)_q$ are given in Theorem \ref{pochs}.
Also, for any $r,m,q \in \N_0$, we obtain the values for $\tbinom{2^rm}{2^rq}$ and $\tbinom{2^rm}{2^rq+1}$ modulo $2$, $4$, $8$ and $16$  
(see Proposition \ref{prop congs}) and modulo some higher powers of 2 (see Propositions \ref{mejora} and \ref{prop cong 2t}).

In the last two sections, we apply the results of Section 4 to study 
central binomial coefficients $c_m=\tbinom{2m}{m}$ and Catalan numbers $C_m$. 
Explicit expressions and recursions for $c_m$ can be found in \eqref{centralsum} -- \eqref{4q2q} and Proposition \ref{prop cm alt}, while mixed expressions between $c_m$'s and integral values of BKP's are given in 
Proposition~\ref{central}. In Proposition \ref{prop catalan} we obtain new recursion formulas for Catalan numbers $C_{2n}$ and $C_{2n+1}$ in terms of $C_0,\ldots,C_n$. 
Finally, we give some congruence relations for $C_m$ modulo 2, 4, 8 and 16.

\section{A reduction formula for binary Krawtchouk polynomials}
As we mentioned in the Introduction, certain properties of binary Krawtchouk polynomials lead to (spectral) geometrical results on $\Z_2^k$-manifolds. Here, in contrast, we will use a geometric result (characters of exterior representations) to obtain a relation between Krawtchouk polynomials. 

\goodbreak 

\subsubsection*{Characters of $p$-exterior representations} 
Let $n=2m$ with $m\in \N$. Consider the special orthogonal group  
$\mathrm{SO}(2m) = \{ A \in M_{2m}(\R)  : AA^t=A^tA=I, \det A=1\}$ of $\R^{2m}$. 
The maximal torus of $\mathrm{SO}(2m)$ is 
$T_{2m} = \{ x(t_1,\ldots,t_m) : t_1,\ldots,t_m \in \R \}$
where 
\begin{equation} \label{x}
x(t_1,\ldots,t_m) = diag(B_1,\ldots,B_m), 
\end{equation}
is the block diagonal matrix with blocks 
$B_i = \left( \begin{smallmatrix} \cos t_i & -\sin t_i \\ \sin t_i & \cos t_i \end{smallmatrix} \right)$ for $i=1,\ldots,m$.

For $0\le p\le m$, let $(\tau_p, \bigwedge^p(\R^{2m})_\C)$ be the $p$-exterior representation of $\mathrm{SO}(2m)$.  
Each $\tau_p$ is irreducible for $0\le p\le m-1$ and $\tau_m$ is the sum of 2 irreducible representations $\tau_m^+$ and $\tau_m^-$ given by the splitting 
$\bigwedge^m(\R^{2m})_\C = \bigwedge^m_+(\R^{2m})_\C \oplus \bigwedge^m_-(\R^{2m})_\C$. Let $\chi_p$ and $\chi_m^\pm$ denote the character of $\tau_p$ and 
$\tau_m^\pm$, respectively.

For $x\in T_{2m}$ there are combinatorial expressions for $\chi_p(x)$, $0\le p \le m$, that we now present 
(see Proposition 3.7 in \cite{MP}). If $I_m = \{1,\ldots, m\}$ and $x=x(t_1,\ldots,t_m)$ then we have
\begin{equation}\label{chipx}
\chi_p (x) = \sum_{\substack{\ell=0 \sk \\(-1)^{\ell+p}=1}}^p \; 2^{\ell} \, \tbinom{m-\ell}{\frac{p-\ell}2} \; \sum_{
\{j_1,\dots,j_\ell\} \subset I_m } \; \Big( \prod_{h=1}^{\ell} \cos t_{j_h} \Big)
\end{equation}
for $0\le p\le m-1$. By duality, $\chi_{2m-p} (x) = \chi_{p} (x)$, we know the characters also for the values $m+1 \le p\le 2m$. 
Furthermore, 
\begin{equation}\label{chimpm}
\chi_m^\pm (x) =  \Bigg( \sum_{\substack{\ell=1  \sk \\ \ell \; \mathrm{odd}}}^m
2^{\ell-1} \,  \tbinom{m-\ell}{\frac{m-\ell}2} \sum_{\{j_1,\dots,j_\ell\} \subset I_m} \Big( \prod_{h=1}^{\ell} \cos t_{j_h} \Big)
\Bigg) \; \pm \; 2^{m-1} i^m \Big( \prod_{j=1}^m \sin t_{j} \Big).
\end{equation}
Note that by \eqref{chimpm}, since $\chi_m^+ + \chi_m^- = \chi_m$, \eqref{chipx} also holds for $p=m$.

\begin{rem}
Clearly, $\chi_n(id)=2^n$. If $x\in T_{2m}$ is of order $2$ then $\chi_n(x)=0$. In fact, $\bigwedge = \bigwedge^0 \oplus \cdots \oplus \bigwedge^m_+ \oplus \bigwedge^m_- \oplus \cdots \oplus \bigwedge^n$. Thus, by using \eqref{pr1} and \eqref{pr3} below, we have
$\chi_n(x) = \sum_{p=0}^n \chi_p(x) = \sum_{p=0}^n K_p^n(2j)=0$ 
for any $1\le j\le m$.
\end{rem}

\subsubsection*{The reduction formula}
By using \eqref{chipx}, we will now express $K_p^{2m}(2j)$ 
as certain integral linear combination of $K_{\ell}^m(j)$ for some alternating indices $\ell$ in $\{0,1,\ldots,p\}$. 
\begin{teo} \label{teo1}
Let $m \in \N$ and $j\in \N_0$. For $0 \le p \le 2m$ we have
\begin{equation} \label{Kraw1}
K_p^{2m}(2j) = \sum_{\substack{\ell=0 \\ \ell \equiv p \,(2)}}^p 2^\ell \, \tbinom{m-\ell}{\f{p-\ell}{2}} \; K_\ell^{m}(j).
\end{equation}
\end{teo}

\begin{proof}
If $C\in \R^{n \times n}$, let $n_C= \dim \, (\R^n)^C$, i.e.\@ the dimension of the space fixed by $C$. 
Let $B$ be a diagonal matrix in $\mathrm{SO}(n)$ of order 1 or 2, that is 
$$B=\mathrm{diag}(\underbrace{-1,\ldots,-1}_{e},\underbrace{1,\ldots,1}_{f})$$
where $\varepsilon_i \in \{\pm 1\}$, $1 \le i \le n$, with an even number of $-1$'s (since $\det(B)=1$). 
If $e_1, \ldots, e_n$ is the canonical basis of $\R^n$, 
put $I_B = \{ 1\le i \le n: Be_i=e_i \}$, hence $n_B=|I_B|$, and $I_B' = \{ 1\le j \le n : j\notin I_B\}$. 
We have that 
\begin{equation} \label{pr1}
\chi_p(B) = K^n_p(n-n_B) = \sum_{\substack{J\subset I_n \\ |J|=p}} (-1)^{|J\cap I_B'|} 
\end{equation}
(see (3.2) and Remark 3.6 in in \cite{MR1}).
Actually, \eqref{pr1} holds for every $B\in \mathrm{SO}(n)$ of order 2, not necessarily diagonal 
(see the proof of Theorem 2.1 in \cite{MPR}).

\sk
Now, let $n=2m$ and let $B$ be any matrix in $\mathrm{SO}(2m)$ of order $\le 2$. Such $B$ is conjugate in $\mathrm{SO}(2m)$ to an element $x_B \in T_{2m}$ as in \eqref{x}, we denote this by $B \sim x_B$. Without loss of generality we can assume that 
$$B=\mathrm{diag}(\underbrace{-1,\ldots,-1}_{e},\underbrace{1,\ldots,1}_{f})$$
with $e+f=2m$ ($e=0$ if and only if $B=Id$). Clearly $B$ is conjugate to $x_B$, where 
$$x_B = x(\underbrace{\pi,\ldots,\pi}_{j},\underbrace{0,\ldots,0}_{m-j}),$$
with $e=2j$ and $f=2(m-j)$. 
By evaluating \eqref{chipx} at $x_B$, we have
\begin{equation} \label{pr2}
\chi_p( x_B) = \sum_{\substack{\ell = 0 \sk \\(-1)^{\ell+p}=1}}^p \; 2^{\ell} \, \tbinom{m-\ell}{\frac{p-\ell}2} \,
\sum_{\substack{J\subset I_m \sk \\ |J|=\ell}} (-1)^{|J \cap I_j|}. 
\end{equation}

Since $B\sim x_B$, the characters of $B$ and $x_B$ coincide and we can equate \eqref{pr1} to \eqref{pr2}. Thus, since $n-n_B=e=2j$, we have
$$K_p^{2m}(2j) = \chi_p(B) = \chi_p( x_B) = \sum_{\substack{\ell = 0 \sk \\(-1)^{\ell+p}=1}}^p \; 2^{\ell} \,
\tbinom{m-\ell}{\frac{p-\ell}2} \,  K_\ell^{m}(j)$$
for every $0\le p \le n=2m$ and $1\le j\le m$.
\end{proof}

Note that expression \eqref{Kraw1} trivially holds for $j \in \Z \smallsetminus \{0,1,\ldots,m\}$ since $K_p^n(j)=0$ for every 
$j<0$ or $j>m$.

\begin{rem} \label{rem1}
Since $0\le \ell\le p\le 2m$, we have $\tbinom{m-\ell}{(p-\ell)/2}=0$ for $\ell>2m-p$ and $K_\ell^m(j)=0$ for $\ell>m$. 
Thus, the upper limit in the sum in Theorem \ref{teo1}, say $\rho$, is actually
the minimum between $p,m$ and $2m-p$, if $p$ and $m$ have the same parity; or the minimum between $p,m-1$ and $2m-p$, otherwise. 
Thus, putting $\mu_p(m)=m$ if $p\equiv m \pmod 2$ and $\mu_p(m)=m-1$ if $p \not\equiv m \pmod 2$,
we have that 
$$\rho = \rho(p,m) = \min\{p,\mu_p(m),2m-p\}$$
is the true upper limit in the summation in \eqref{Kraw1}. 
\end{rem}

By considering the cases $p$ even or odd separately, we get simpler expressions for \eqref{Kraw1} as follows
\begin{equation} \label{Kraw1b}
\begin{split}
K_{2q}^{2m}(2j) & = \sum\limits_{k=0}^q 4^k \; \tbinom{m-2k}{q-k} \; K_{2k}^m(j)  
\bsk \\
K_{2q+1}^{2m}(2j) & = 2 \, \sum\limits_{k=0}^q 4^k \; \tbinom{m-2k-1}{q-k} \; K_{2k+1}^m(j)  
\end{split}
\end{equation}
with $q\le m$ in the even case and $q\le m-1$ in the odd case.

\sk

There are 3 basic symmetry relations between binary Krawtchouk polynomials; namely, 
$K_k^n(n-k) = K_{n-k}^n(k)$, $K_k^n(j)=(-1)^j K_{n-k}^n(j)$ and  
\begin{equation} \label{sym}
\tbinom nj K_k^n(j)= \tbinom nk K_j^n(k).
\end{equation}
By using \eqref{sym} we can get an expression for $K_{2j}^{2m}(p)$ similar to \eqref{Kraw1}.

\begin{coro} \label{coro alt}
For $0\le j,p \le m \in \N$ we have
\begin{equation} \label{kraw alt1} 
K_{2j}^{2m}(p) = \tf{\binom{2m}{2j}}{\binom{2m}{p} \binom mj} 
\sum_{\substack{\ell=0 \\ \ell \equiv p \,(2)}}^p 2^\ell \, \tbinom{m-\ell}{\f{p-\ell}{2}} \tbinom{m}{\ell}\; K_j^{m}(\ell).
\end{equation}
\end{coro}

\begin{proof}
The result follows directly by first applying 
\eqref{sym} to $K_{2j}^{2m}(p)$ and then using \eqref{Kraw1} and \eqref{sym} again. 
\end{proof}

\begin{rem}
With expressions \eqref{Kraw1} and \eqref{kraw alt1}, we can recursively compute 
$K_p^{2m}(j)$, by using BKP's of order $2m$, for $j$ even and any $p$ or for $p$ even and any $j$. It would remain to cover the cases 
$K^{2m}_{\textrm{odd}}(\textrm{odd})$, which are only $m^2$ cases out of $(2m+1)^2$. Hence, asymptotically, we cover 75\% of the cases, as one could expect a priori. The other symmetry relations are of no help for this matter.
\end{rem}

As a direct consequence of the previous result we have the following cancellation property for Krawtchouk polynomials.
\begin{coro} \label{coro1}
Let $m,j\in \N$. For $0 \le p \le m$ we have
\begin{equation} \label{Kraw2}
\sum_{p=0}^{2m} \sum_{\substack{\ell=0 \\ \ell\equiv p \,(2) }}^p 2^\ell \, \tbinom{m-\ell}{\f{p-\ell}{2}} \; K_\ell^m(j) 
= 2^m \sum_{\ell=0}^{m} K_\ell^m(j)= 0.
\end{equation}
\end{coro}

\begin{proof}
In \cite{MPR} we have shown that 
\begin{equation} \label{pr3}
K_0^n(j) + K_1^n(j) +\cdots + K_n^n(j) = 0
\end{equation}
for every $1\le j\le n$, and hence the second identity in \eqref{Kraw2} follows. 
To get the first identity in \eqref{Kraw2} we apply \eqref{pr3} to \eqref{Kraw1}. 
Note that $\sum\limits_{p=0}^{2m} \sum\limits_{\ell=0}^p = \sum\limits_{\ell=0}^{2m} \sum\limits_{p=\ell}^{2m}$ 
and, since $K_k^n(j)=0$ for $j>n$ and $\tbinom{m-\ell}{\f{p-\ell}{2}}=0$ for $\f{p-\ell}{2} > m-\ell$, we have
$$\sum_{p=0}^{2m} \sum_{\substack{\ell=0 \\ \ell\equiv p \,(2) }}^p 2^\ell \, \tbinom{m-\ell}{\f{p-\ell}{2}} \; K_\ell^m(j) = 
\sum_{\ell=0}^{m}  2^\ell \, \Bigg( \sum_{\substack{p=\ell \\ \ell\equiv p \,(2) }}^{2m-\ell} \tbinom{m-\ell}{\f{p-\ell}{2}} \Bigg) \; 
K_\ell^m(j).$$
Finally, since
$$\sum_{\substack{p=\ell \\ \ell\equiv p \,(2) }}^{2m-\ell} \tbinom{m-\ell}{\f{p-\ell}{2}} = 
\sum_{k=0}^{m-\ell} \tbinom{m-\ell}{k} = 2^{m-\ell}$$ 
the left hand side of \eqref{Kraw2} equals $2^m \sum_{\ell=0}^{m} K_\ell^m(j)$,
as desired.
\end{proof}

\begin{ejem}
We now illustrate the big cancellations present in the corollary. 
Let $m=4$ and $j=3$. The terms of the inner sum of the left hand side of \eqref{Kraw2} are 
\begin{equation*}
\begin{aligned}
& p=0,1,7,8 & \qquad & 2^0 \tbinom 40 K_0^4(3), \quad 2^1 \tbinom 30 K_1^4(3), \quad  2^1 \tbinom 33 K_1^4(3), 
\quad  2^0 \tbinom 44 K_0^4(3) & \\
& p=2,6 & \qquad & 2^0 \tbinom 41 K_0^4(3) + 2^2 \tbinom 20 K_2^4(3), \quad 2^0 \tbinom 43 K_0^4(3) + 2^2 \tbinom 22 K_2^4(3) & \\
& p=3,5 & \qquad & 2^1 \tbinom 31 K_1^4(3) + 2^3 \tbinom 10 K_3^4(3), \quad 2^1 \tbinom 32 K_1^4(3) + 2^3 \tbinom 11 K_3^4(3) & \\
& p=4 & \qquad & 2^0 \tbinom 42 K_0^4(3) + 2^2 \tbinom 21 K_2^4(3) + 2^4 \tbinom 00 K_4^4(3) & \\
\end{aligned}
\end{equation*} 
respectively. 
The sum of all these terms is
$$2^0 S_4 \, K_0^4(3) + 2^1 S_3 \, K_1^4(3) + 2^2 S_2 \, K_2^4(3) + 
2^3 S_1 \,K_3^4(3) + 2^4 S_0 \, K_4^4(3).$$
where $S_j = \tbinom j0 + \cdots + \tbinom jj$ 
for $0\le j\le 4$.
Thus, the left hand side of \eqref{Kraw2} is 
$$\sum_{p=0}^8 \sum_{\substack{\ell=0 \\ \ell\equiv p \,(2) }}^p 2^\ell \, \tbinom{4-\ell}{\f{p-\ell}{2}} \; K_\ell^4(3) =  
2^4 \big( K_0^4(3) + K_1^4(3) + K_2^4(3) + K_3^4(3) + K_4^4(3) \big) =0 $$
since the values for $K_i^4(3)$ with $i=0,\ldots,4$ are $1,-2,0,2$ and $-1$ respectively. 
\end{ejem}

\section{A family of identities for BKP's}
Formula \eqref{Kraw1} can be iterated to obtain (more involved) expressions for $K_p^{2^r m}(2^s j)$ 
in terms of $K_p^m(j)$'s for different $p$'s. 
For instance, for $r=s=2$ and $j\le m$ odd, we have
\begin{eqnarray*}
K_p^{4m}(4j) &=& \sum_{\substack{\ell=0 \\ \ell \equiv p \,(2)}}^p 2^\ell \, \tbinom{2m-\ell}{\f{p-\ell}{2}} \; K_\ell^{2m}(2j) 
= \sum_{\substack{\ell=0 \\ \ell \equiv p \,(2)}}^p 2^\ell \, \tbinom{2m-\ell}{\f{p-\ell}{2}} \sum_{\substack{k=0 \\ k \equiv \ell\,(2)}}^\ell 2^k \, \tbinom{m-k}{\f{\ell-k}{2}} \; K_k^m(j) ,
\end{eqnarray*}   
where we have applied \eqref{Kraw1} twice.
That is,
\begin{equation} \label{iter1}
K_p^{4m}(4j) = \sum_{\substack{0\le k \le \ell \le p \\ k\equiv \ell \equiv p \,(2)}} 2^{k+\ell} \, \tbinom{2m-\ell}{\f{p-\ell}{2}} \, \tbinom{m-k}{\f{\ell-k}{2}} \; K_k^m(j).
\end{equation}
However, for $K_p^{4m}(2j)$ we can apply \eqref{Kraw1} only once.

\subsubsection*{A generalized reduction formula}
By continuing with the previous process, we can express the values of $K_p^{2^r m}(x)$ for $x$ even as linear combinations of values of some BKP's of much smaller orders, thus generalizing Theorem \ref{teo1}. 
We will need the following notation. For $r,s\in \N$, with $r$ fixed, put 
\begin{equation} \label{fsr}
f(s,r) := (r-s) \, \chi_{_{[1,r]}}(s) = \begin{cases} r-s & \qquad \text{if } s \le r, \\ 0 & \qquad \text{if } s > r, \end{cases}
\end{equation} 
where $\chi_{_{[1,r]}}$ is the characteristic function of the interval $[1,r]$. Note that $f(t,t)=0$ for every $t$.

\begin{teo} \label{teo2}
Let $m,p,r,s\in \N$, $j\in \N_0$ such that 
$2(\nu-1) \le p \le 2^r m$ and $2^s j \le 2^r m$ where $\nu = \min\{r,s\}$.
Then, we have 
\begin{equation} \label{Kraw3c}
K_p^{2^r m}(2^s j) = \sum_{\substack{0\le p_\nu \le \cdots \le p_1 \le p \sk \\  p_\nu 
\equiv \cdots \equiv p_1\equiv p \,(2)}}  
\, 2^{p_1 + \cdots + p_\nu} \left( \prod_{k=1}^{\nu} \tbinom{2^{r-k} m - p_k}{\f{p_{k-1}-p_k}{2}}  \right) \: 
K_{p_\nu}^{2^{f(s,r)}m}(2^{f(r,s)}j) 
\end{equation}
where we use the convention $p_0 = p$.
\end{teo}

\noindent
\textit{Note.} If $r=s$ there is a big simplification in \eqref{Kraw3c}, since in this case 
$K_{p_\nu}^{2^{f(s,r)}m}(2^{f(r,s)}j)$ equals $K_{p_r}^{m}(j)$ for $m,j$ odd. By taking $r=s=1$ we get Theorem \ref{teo1}.

\begin{proof}
We will apply Theorem \ref{teo1} the maximum possible number of times, which is $\nu$. So, we will consider the cases $s\le r$ and 
$s>r$ separately. 

\sk (i) Assume first that $s\le r$. Hence we can apply \eqref{Kraw1} just $s$ times. We proceed by induction on $r$ using 
\eqref{Kraw1}. The first step in the induction, i.e.\@ $r=2$ ($r=1$ is just Theorem \ref{teo1}), is done in the observation before the statement. 
For the inductive step, suppose first that $s=r$.
By considering 
$m'=2^{r-1}m$ and $j'=2^{r-1} j$, we can apply Theorem~\ref{teo1} and get
$$K_p^{2^rm}(2^rj) = K_p^{2m'}(2j') = \sum_{\substack{0\le \ell \le  p \sk \\  \ell \equiv p \,(2)}} 2^\ell \, \tbinom{m'-p}{\f{p-\ell}2} \; 
K_\ell^{2^{r-1}m}(2^{r-1}j).$$ 
Thus, by the inductive hypothesis, 
$$K_p^{2^rm}(2^rj) = \sum_{\substack{0\le \ell \le  p \sk \\  \ell \equiv p \,(2)}} 2^\ell \, \tbinom{m'-p}{\f{p-\ell}2} \, 
\sum_{\substack{0\le q_{r-1} \le \cdots \le q_1 \le \ell \sk \\  q_{r-1} \equiv  \cdots \equiv q_{1} \equiv p \;(2)}}  \, 
2^{q_1 + \cdots + q_{r-1}} \prod_{k=1}^{r-1} \tbinom{2^{r-1-k} m - q_k}{\tf{q_{k-1} - q_k}{2}}  \:   K_{q_{r-1}}^{m}(j).$$
By renaming the indices as follows $p_1=\ell$, $p_2=q_1$, $\ldots,$ $p_r=q_{r-1}$, we get 
\begin{equation} \label{Kraw3d}
K_p^{2^r m}(2^r j) = \sum_{\substack{0\le p_r \le \cdots \le p_1 \le p \sk \\  p_r \equiv \cdots \equiv p_1\equiv p \,(2)}}  
\, 2^{p_1 + \cdots + p_r} \left( \prod_{k=1}^r \tbinom{2^{r-k} m - p_k}{\f{p_{k-1}-p_k}{2}}  \right) \:   K_{p_r}^m(j) ,
\end{equation}
which equals expression \eqref{Kraw3c} with $\nu=s=r$ and $f(r,r)=0$.

For the case when $s<r$, we proceed similarly as before. By applying \eqref{Kraw1} a number $s$ of times, we get
\begin{equation} \label{Kraw3}
K_p^{2^r m}(2^s j) = \sum_{\substack{0\le p_s \le \cdots \le  p_2 \le p_1 \le p \sk \\  p_s 
\equiv \cdots \equiv p_1\equiv p \!\!\!\pmod 2}}  
\, 2^{p_1 + \cdots + p_s} \left( \prod_{k=1}^{s} \tbinom{2^{r-k} m - p_k}{\f{p_{k-1}-p_k}{2}}  \right) \:   K_{p_s}^{2^{r-s}m}(j) ,
\end{equation}
which equals \eqref{Kraw3c} since $\nu=s$, $f(s,r)=r-s$ and $f(r,s)=0$, in this case.

\sk 
(ii) When $s>r$, we can apply \eqref{Kraw1} a number $r$ of times. By proceeding similarly as before, we get 
\begin{equation} \label{Kraw3b}
K_p^{2^r m}(2^s j) = \sum_{\substack{0\le p_r \le \cdots \le  p_2 \le p_1 \le p \sk \\  p_r 
\equiv \cdots \equiv p_1\equiv p \!\!\!\pmod 2}}  
\, 2^{p_1 + \cdots + p_r} \left( \prod_{k=1}^{r} \tbinom{2^{r-k} m - p_k}{\tf{p_{k-1}-p_k}{2}}  \right) \:   K_{p_r}^{m}(2^{s-r}j).
\end{equation}

\sk 
It is now clear that, by using \eqref{fsr}, the expressions obtained in \eqref{Kraw3d} -- \eqref{Kraw3b} take the single form in \eqref{Kraw3c}, 
and the result thus follows. 
Finally, note that we need to consider $p\ge 2(\nu-1)$, for if not the summation set would be empty. 
\end{proof}

Typically, one will apply \eqref{Kraw3c} when $j$ and $m$ are odd, for if not we can keep absorbing the powers of 2. However, we also want to consider the case $j=0$.

\begin{rem} \label{rem2}
There are some redundancies in \eqref{Kraw3c} because some terms can vanish. A more accurate, though complicated, expression can be obtained by taking into account the exact limits of summation (see Remark \ref{rem1}). In this case we have
$$\sum_{\substack{0\le p_\nu \le \cdots \le p_1 \le p \sk \\  p_\nu 
\equiv \cdots \equiv p_1\equiv p \,(2)}} \quad \rightsquigarrow \quad 
\sum\limits_{\substack{p_\nu=0 \\ p_\nu\equiv p \,(2)}}^{\rho_{\nu-1}} \cdots \sum_{\substack{p_2=0 \\ 
p_2\equiv p \,(2)}}^{\rho_1} \sum_{\substack{p_1=0 \\ p_1\equiv p \,(2)}}^{\rho}$$
where the symbol $\rightsquigarrow$ means that the first summation (which is easier to write) should be replaced by the second summation
(which looks more complicated but involves much less terms) and where 
$$\rho_i = \min\{ \, p_{i-1}, \, \mu_{p_{i-1}}(2^{r-i+1}m), \, 2^{r-i}m-p_{i-1} \, \}$$
for $i=0,\ldots,\nu-1$ with the conventions $\rho_0=\rho$ and $p_0=p$.
\end{rem}

\subsubsection*{Explicit computations}
We will now illustrate the formulas obtained in Theorems \ref{teo1} and \ref{teo2} with some examples. 
It will be helpful to have the values $K_k^n(j)$, $0\le i,j \le n$, for $n=1,2,\ldots,8$ at hand. Thus, we present them as matrices 
$K_n = (K_{ij}^n)$ with $K_{ij}^n= K_i^n(j)$ in Table \ref{tablitas} below. 

\begin{table}[ht] 
\caption{The values of $K_p^n(j)$ for $1\le n\le 8$ and $0\le p,j \le n$}
\hrule \hrule
{\scriptsize
\begin{gather*}
K_1 = 
\left( \begin{array}{rr}
1& 1  \\
1& -1 \end{array} \right)
\qquad \qquad 
K_2 = 
\left( \begin{array}{rrr}
1& 1& 1 \\
2& 0&-2 \\
1& -1& 1
\end{array} \right)
\qquad \qquad 
K_3 = \left( \begin{array}{rrrr}
1& 1& 1& 1 \\
3& 1& -1& -3 \\
3& -1& -1& 3 \\
1& -1& 1& -1
\end{array} \right) \\
K_4 = 
\left( \begin{array}{rrrrr}
1& 1& 1& 1& 1 \\
4& 2& 0 &-2&-4 \\
6& 0 &-2& 0 & 6 \\
4&-2& 0 & 2&-4 \\
1&-1& 1&-1& 1
\end{array} \right) 
\qquad \qquad 
K_5 = 
\left( \begin{array}{rrrrrr}
1& 1& 1& 1& 1& 1 \\
5& 3& 1&-1&-3&-5 \\
10& 2&-2& -2&2& 10 \\
10& -2&-2& 2&2& -10 \\
5& -3& 1&1&-3&5 \\
1&-1& 1&-1& 1&-1
\end{array} \right)
\end{gather*}}
{\tiny
\begin{gather*}
K_6 = 
\left( \begin{array}{rrrrrrr}
1& 1& 1& 1& 1 &1 &1 \\
6& 4& 2& 0 &-2&-4 &-6 \\
15 & 5 & -1 & -3 & -1 & 5 & 15 \\
20& 0 & -4& 0 & 4& 0 & -20 \\
15 & -5 & -1 & 3 & -1 & -5 & 15 \\
6& -4& 2& 0 & 2 & 4 &-6 \\
1& -1& 1& -1& 1 &-1 &1
\end{array} \right) 
\qquad K_7 = 
\left( \begin{array}{rrrrrrrr}
1& 1& 1& 1& 1 &1 &1 &1 \\
7& 5& 3& 1&-1&-3 &-5&-7 \\
21 & 9 & 1 & -3 & -3 & 1 & 9 & 21 \\
35& 5& -5& -3&3& 5&-5 & -35 \\
35& -5& -5& 3&3& -5&-5 & 35 \\
21 & -9 & 1 & 3 & -3 &-1 & 9 &-21 \\
7& -5& 3& -1&-1&3 &-5&7 \\
1& -1& 1& -1& 1 &-1 &1 &-1
\end{array} \right) 
\\ 
K_8 = 
\left( \begin{array}{rrrrrrrrr}
1& 1& 1& 1& 1& 1& 1 &1 &1 \\
8& 6& 4& 2& 0 &-2&-4 &-6 & -8\\
28 & 14& 4 & -2 & -4 & -2 & 4 & 14 & 28\\
56 & 14& -4 & -6 & 0 & 6 & 4 & -14 & -56\\
70 & 0 & -10 & 0 & 6& 0 &-10& 0 & 70 \\
56 & -14& -4 & 6 & 0 & -6 & 4 & 14 & -56\\
28 & -14& 4 & 2 & -4 & 2 & 4 & -14 & 28\\
8& -6& 4& -2& 0 & 2 & -4 & 6 & -8\\
1& -1& 1& -1& 1& -1& 1 &-1 &1
\end{array} \right) 
\end{gather*}}
\hrule \hrule
\label{tablitas}
\end{table}

\sk
Note that the sum of the even rows (i.e., the odd numbered ones) vanish; that is $\sum_{j=0}^n K_i^n(j)=0$ for $i$ odd. This is a general fact 
proved in \cite{ChS}, see (5.1) -- (5.5). On the other hand, expression \eqref{pr3} accounts for the vanishing of the sum of the columns 
(even or odd, except for the first one).

Observe that in the even matrices $K_2$, $K_4$, $K_6$, $K_8$, there are several entries divisible by $2$, $4$ and $8$. 
In the last two sections we will show that this is indeed a general phenomenon.

\begin{ejem}
Here we compute the single value $K_2^8(4)$ using 3 different expressions: 
the definition \eqref{eq. kraws} and Theorems \ref{teo1} and \ref{teo2}. 
We have 
\begin{align*} 
& K_2^8(4) = \tbinom 40 \tbinom 42 - \tbinom 41 \tbinom 41 + \tbinom 42 \tbinom 40 = 6-16+6=-4,    \\
& K_2^8(4) = 2^0\tbinom 41 K_0^4(2) + 2^2 \tbinom 20 K_2^4(2) = 4+4(-2)=-4,  \\ 
& K_2^8(4) = \tbinom 41 \tbinom 20 K_0^2(1) +  2^2\tbinom 20 \tbinom 21 K_0^2(1) + 
2^4 \tbinom 20 \tbinom 00 K_2^2(1) = 4+8-16=-4, 
\end{align*}
by \eqref{eq. kraws}, \eqref{teo1} and \eqref{teo2}, respectively.
\end{ejem}

\begin{ejem}
We will compute the values $K_4^8(2j)$, $0 \le j \le 4$. 
By \eqref{k0k1}, 
%Since $K_p^n(0)=\binom np$ and $K_p^n(n) = (-1)^p \binom np$ 
we have that $K_4^8(0) = K_4^8(8) = \tbinom 84 = 70$. 
By Theorem \ref{teo1}, for $j=1,3$ we have
$$K_4^8(2j) = \sum_{\ell=0,2,4} 2^\ell \tbinom{4-\ell}{\f{4-\ell}{2}} K_\ell^4(j) = 
\tbinom 42 K_0^4(j) + 2^2 \tbinom 21 K_2^4(j) + 2^4 \tbinom 00 K_4^4(j) = -10 ,$$ 
where we have used the second and fourth columns of $K_4$ in Table \ref{tablitas}.

Let us now compute $K_4^8(4)$ with Theorem \ref{teo2}. By taking $m=2$ and $j=1$ in \eqref{iter1} we have
$$K_4^8(4) = \sum_{\substack{0\le k\le \ell \le 4 \sk \\ k, \ell \text{ even}}} 2^{k+\ell} \tbinom{4-\ell}{\f{4-\ell}{2}} 
\tbinom{2-k}{\f{\ell-k}{2}} \; K_k^2(1).$$
We only have to sum over the pairs $(k,\ell)$ of the form $(0,0)$, $(0,2)$, $(2,2)$ and $(0,4)$
since $(2,4)$ and $(4,4)$ do not contribute to the sum because of the appearance of $\tbinom 01=0$ and 
$K_4^2(1)=0$ in the corresponding terms. Thus, we get
\begin{eqnarray*}
K_4^8(4) &=& \tbinom 42 \tbinom 20 K_0^2(1) + 2^2 \tbinom 21 \tbinom 21 K_0^2(1) +
2^4 \tbinom 21 \tbinom 00  K_2^2(1) + 2^4 \tbinom 00 \tbinom 22 K_0^2(1)  \\ 
&=& (6+16+16) K_0^2(1) + 32 K_2^2(1) = 38 - 32 = 6 \, 
\end{eqnarray*}
where we have used (the second column of) $K_2$ in this case.

The computations we have obtained for $K_4^8(2j)$ are in coincidence with the values in (the fifth row of) $K_8$. 
Similarly, one can compute $K_p^8(2j)$ for any $0\le p\le 8$, $p\ne 4$. 
\end{ejem}

\sk
We now present a more involved example to show how Theorem \ref{teo2} really works.

\begin{ejem}
Let us compute $K_6^{48}(40)=K_6^{2^4 3}(2^3 5)$. So $m=3$, $r=4$ and $s=3$. 
Thus, $\nu=3$, and by \eqref{Kraw3c} we have 
\begin{equation} \label{ej48}
K_p^{48}(40) = \sum_{\substack{0\le p_3 \le p_2 \le p_1 \le 6 \sk \\  
p_3, p_2, p_1 \text{ even}}}  \, 
2^{p_1 + p_2 + p_3} \cdot \pi(p_3,p_2,p_1)  \cdot K_{p_3}^{6}(5) 
\end{equation}
since $f(3,4)=1$ and $f(4,3)=0$,
where we have used the notation
\begin{equation} \label{pi321} 
\pi(p_3,p_2,p_1) := \prod_{k=1}^{\nu} \tbinom{2^{4-k} 3 - p_k}{\f{p_{k-1}-p_k}{2}}
= \tbinom{24 - p_1}{\f{6-p_1}{2}} \tbinom{12-p_2}{\f{p_{1}-p_2}{2}} 
\tbinom{6 - p_3}{\f{p_{2}-p_3}{2}}.
\end{equation} 
In this way, $K_6^{48}(40)$ is only expressed in terms of $K_p^6(5)$ with $p$ even, which we know have the values $K_0^6(5)=K_6^6(5)=1$, 
$K_2^6(5)=2$ and $K_4^6(5)=-5$, by Table~1.

The 20 allowed triplets in \eqref{ej48} are 
$(0,0,0)$, $(0,0,2)$, $(0,0,4)$, $(0,2,2)$, $(0,0,6)$, $(0,2,4)$, $(2,2,2)$, $(0,2,6)$, $(0,4,4)$, $(2,2,4)$, $(0,4,6)$, $(2,2,6)$, $(2,4,4)$, $(0,6,6)$, $(2,4,6)$, $(4,4,4)$, $(2,6,6)$, $(4,4,6)$ and $(6,6,6)$.
Therefore, we have
\begin{multline*}
K_6^{48}(40) = \pi(0,0,0) + 2^2 \pi(0,0,2) + 2^4 \big(\pi(0,0,4)+\pi(0,2,2)\big) \\ 
              + 2^6 \big(\pi(0,0,6)+\pi(0,2,4)\big) + 2^8 \big(\pi(0,2,6)+\pi(0,4,4)\big) \\ 
							+ 2^{10} \pi(0,4,6) + 2^{12} \pi(0,6,6) + 2^{18} \pi(6,6,6) \\ 
							+ 5 \Big\{ 2^6 \pi(2,2,2) + 2^8  \pi(2,2,4) + 2^{10} \big(\pi(2,2,6) + \pi(2,4,4) \big) \\ 
							+ 2^{12} \big( \pi(2,4,6) - \pi(4,4,4) \big) 	- 2^{14} \pi(4,4,6) - 2^{16} \pi(4,6,6) \Big\}
\end{multline*}
where all the $\pi(p_3,p_2,p_1)$ can be easily computed by using \eqref{pi321}.
\end{ejem}

\section{Applications to binomial coefficients}
As a direct consequence of Theorem \ref{teo2} we get basic identities between binomial coefficients, 
expressing $\binom{2^rm}k$ in terms of numbers $\binom{2^tm}j$ with $t<r$ for some $j$'s. 
In fact, under the assumptions of Theorem \ref{teo2}, taking $j=0$ in \eqref{Kraw3c} and using that $K_p^n(0) = \binom np$, 
for any $1\le s,r$ we have
\begin{equation} \label{comb3}
\tbinom{2^r m}{p} = 
\sum_{\substack{0\le p_\nu \le \cdots \le  p_2 \le p_1 \le p \sk \\  p_\nu 
\equiv \cdots \equiv p_1\equiv p \!\!\!\pmod 2}}  \, 2^{p_1 + \cdots + p_\nu}
\left( \prod_{k=1}^{\nu} \tbinom{2^{r-k} m - p_k}{\f{p_{k-1}-p_k}{2}}  \right) \: 
\tbinom{2^{f(s,r)}m}{p_\nu} 
\end{equation}
with $\nu =\min\{r,s\}$, where $\tbinom{2^{f(s,r)}m}{p_\nu}$ equals $\tbinom{2^{r-s}m}{p_\nu}$ or $\tbinom{m}{p_\nu}$ 
depending on whether $s< r$ or $s \ge r$. 
A much simpler expression is obtained when $s=1$, namely 
\begin{equation} \label{comb4}
\tbinom{2^r m}{p} = \sum_{\substack{0\le \ell \le p \sk \\ \ell \equiv p \, (2)}}  
\, 2^\ell \tbinom{2^{r-1}m - \ell}{\tf{p-\ell}{2}}  \:   \tbinom{2^{r-1}m}{\ell}.
\end{equation}
For the simplest case, i.e.\@ $r=s=1$, we can give more explicit expressions.

\begin{lema} \label{prop1}
For any integers $0\le q \le m$ we have
\begin{equation} \label{comb1}
\begin{split} 
\tbinom{2m}{2q} & = \sum_{j=0}^q 4^j \, \tbinom{m-2j}{q-j} \tbinom{m}{2j} 
= \sum\limits_{j=0}^q 4^j \, \tbinom{m}{q+j} \tbinom{q+j}{2j}, \bsk \\
\tbinom{2m}{2q+1} & = 2\sum_{j=0}^q 4^j \, \tbinom{m-2j-1}{q-j} \tbinom{m}{2j+1} = 
2\sum\limits_{j=0}^q 4^j \, \tbinom{m}{q+j+1} \tbinom{q+j+1}{2j+1}, 
\end{split}
\end{equation} 
where $q<m$ in the second identity.
\end{lema}

\begin{proof}
By taking $j=0$ in \eqref{Kraw1b}, or taking $r=1$ in \eqref{comb4}, and considering the cases $p=2q$ and $p=2q+1$, 
we get the first equalities in each of the expressions in \eqref{comb1}.
To see the remaining equalities, notice that 
$$\tbinom{m-2j}{q-j} \tbinom{m}{2j} = \tbinom{m}{2j} \tbinom{m-2j}{q+j-2j} = \tbinom{m}{q+j} \tbinom{q+j}{2j},$$
where in the second equality we have applied the relation 
$\tbinom rs \tbinom st = \tbinom rt \tbinom{r-t}{s-t}$ with $r=m$, 
$s=q+j$ and $t=2j$. Thus, we get the second equality in the first row of \eqref{comb1}. Proceeding similarly for the odd case, one gets the desired expression in the statement.
\end{proof}

\begin{rem}
By Pascal's identity, one can get similar formulas for $\binom{2m+1}{2q+1}$ and $\binom{2m+1}{2q}$ as in Lemma~\ref{prop1}, 
by combining the expressions in \eqref{comb1}. 
\end{rem}

\subsection{Recursions}
Next, we will give alternative expressions for $\tbinom{2m}{2q}$, $\tbinom{2m}{2q+1}$, $\tbinom{2m+1}{2q}$ and $\tbinom{2m+1}{2q+1}$ in terms of 
$\tbinom mq$. We will need to make use of the \textit{double factorial} of $n$
$$n!! = n(n-2)(n-4) \cdots $$ 
i.e.\@ $n!!=\prod_{k=0}^{m} (n-2k)$ with $m = \lceil \tfrac n2 \rceil -1$, and the \textit{Pochhammer symbol} 
$$(n)_j = \prod_{k=0}^{j-1} (n-k) = n(n-1)(n-2)\cdots (n-j+1) $$  
also known as \textit{falling factorial}.

\begin{teo} \label{pochs}
If $q,m \in \N_0$ then we have
\begin{align} 
\label{b1}
\tbinom{2m}{2q} &= 
\tbinom mq  \, \sum_{j=0}^q \tf{2^j}{j!(2j-1)!!} \, (q)_j (m-q)_j, 
\\ 
\label{b2}
\tbinom{2m}{2q+1} &= 
2(m-q) \tbinom{m}{q} \, \sum_{j=0}^q \tf{2^j}{j!(2j+1)!!} \, (q)_j (m-q-1)_j, 
\\
\label{b3}
\tbinom{2m+1}{2q} & = 
(2q+1) \tbinom mq  \, \sum_{j=0}^q \tf{2^j}{j!(2j+1)!!} \, (q)_j (m-q)_j, \\ 
\label{b4}
\tbinom{2m+1}{2q+1} & = 
(2(m-q)+1) \tbinom{m}{q} \, \sum_{j=0}^q \tf{2^j}{j!(2j+1)!!} \, (q)_j (m-q)_j, 
\end{align}
where the sums are in $\Q$.
\end{teo}

\begin{proof}
Our starting point will be the expressions for $\tbinom{2m}{2q}$ and $\tbinom{2m}{2q+1}$ in \eqref{comb1}. First, notice that 
\begin{equation} \label{2j!}
(2j)! = 2^j \, j! (2j-1)!!
\end{equation} 
Thus, the general term in the first summation for $\tbinom{2m}{2q}$ can be written as follows 
$$4^j \tbinom{m-2j}{q-j} \tbinom{m}{2j} = 4^j \, \tf{(m-2j)!}{(q-j)!(m-q-j)!} \, \tf{m!}{(2j)!(m-2j)!} 
      = \tf{2^j}{j!(2j-1)!!} \, \tf{m!}{(q-j)!(m-q-j)!}$$
and hence, we get 
$$\tbinom{2m}{2q} = m! \sum_{j=0}^q \tf{2^j}{j!(2j-1)!!} \, \big( (q-j)!(m-q-j)! \big)^{-1} .$$
After performing the sum of the fractions involved, we get
\begin{eqnarray*}
\tbinom{2m}{2q} &=& \tf{m!}{q!(m-q)!} \, \sum_{j=0}^q \tf{2^j}{j!(2j-1)!!} \, \Big( \prod_{k=0}^{j-1} (q-k)(m-q-k) \Big) \\
   &=& \tbinom mq \, \sum_{j=0}^q \tf{2^j}{j!(2j-1)!!} \, \Big( \prod_{k=0}^{j-1} (q-k) \Big) \Big( \prod_{k=0}^{j-1} (m-q-k) \Big), 
\end{eqnarray*}
and by using the Pochhammer symbols we finally obtain \eqref{b1}.

For $\tbinom{2m}{2q+1}$ we proceed similarly as before. By \eqref{2j!}, 
\begin{equation} \label{2j+1!}
(2j+1)!=(2j+1)(2j)!= 2^j j! (2j+1)!!
\end{equation}
and hence, from the expression in the second row of \eqref{comb1} we get
\begin{eqnarray*}
\tbinom{2m}{2q+1} &=&  2 m! \sum_{j=0}^q \tf{2^j}{j!(2j+1)!!} \, \big((q-j)!(m-q-1-j)!\big)^{-1} \\
  &=& \tf{2m!}{q!(m-1-q)!} \, \sum_{j=0}^q \tf{2^j}{j!(2j+1)!!} \, \Big( \prod_{k=0}^{j-1} (q-k) \Big) \Big( \prod_{k=0}^{j-1} (m-q-1-k) \Big) 
\end{eqnarray*}
from which \eqref{b2} readily follows.

Finally, since $\binom{2m+1}{2q} = \binom{2m}{2q-1} + \binom{2m+1}{2q}$ and $\binom{2m+1}{2q+1} = \binom{2m}{2q} + \binom{2m}{2q+1}$, expressions \eqref{b3} and \eqref{b4} follow directly from \eqref{b1} and \eqref{b2}, after some tedious but straightforward computations. 
\end{proof}

\begin{rem}
(i) Theorem \ref{pochs} gives expressions for $\tbinom{2m+\varepsilon}{2q+\varepsilon'} / \tbinom mq$, 
with $\varepsilon, \varepsilon' \in \{0,1\}$. 

(ii) It is known that  
$(q)_j = \sum_{i=0}^j (-1)^{j-i} \, s(j,i) \, q^i$.
Hence, the Theorem \ref{pochs} relates binomial coefficients with Stirling numbers of the first kind 
(see A048994 in \cite{OEIS}). % in an intricate way. 
For instance, \eqref{b1} takes the form
\begin{equation} \label{stirling}
\tbinom{2m}{2q} = \tbinom mq  \, \sum_{j=0}^q \tf{2^j}{j!(2j-1)!!} \, \sum_{k,l=0}^j (-1)^{k+l} \, s(j,k) s(j,l) \, q^k(m-q)^l. 
\end{equation}
\end{rem}

Theorem \ref{pochs} implies an expression for any product of $m-q$ consecutive odd 
(resp.\@ even) positive numbers in terms of fractions of factors of small order. 
\begin{coro} \label{prodimp}
For $q\le m-1$, with the convention $(-1)!!=1$, the product of $m-q$ consecutive odd numbers is 
$$N:=\prod_{j=q}^{m-1} (2j+1) = (2(m-q)-1)!! \sum_{j=0}^q \f{2^j}{j!(2j-1)!!} \, (q)_j (m-q)_j$$
and hence the product of $m-q$ consecutive even numbers is  
$$M:=\prod_{j=q}^{m-1} (2j) = \tfrac{(2m-1)!}{(2q-1)!N}.$$
\end{coro}

\begin{proof}
First note that, by \eqref{2j!}, we have 
$$\tbinom{2m}{2q} = \tf{(2m)!}{(2q)!(2(m-q))!} = \tf{2^m m! (2m-1)!!}{2^q q!(2q-1)!! 2^{m-q} (m-q)! (2(m-q)-1)!!\, } = 
\tbinom mq \, \tf{(2m-1)!!}{(2(m-q)-1)!!\, (2q-1)!!}.$$  
Comparing this with \eqref{b1} we have 
$$\tf{(2m-1)(2m-3) \cdots (2q+1)}{(2(m-q)-1)!!} = \sum_{j=0}^q \tf{2^j}{j!(2j-1)!!} \, (q)_j (m-q)_j$$
from which the expression for $N$ follows. Since $MN=\tfrac{(2m-1)!}{(2q-1)!}$ we are done.  
\end{proof}

\begin{ejem}
We want to compute $N=13\cdot 15 \cdot 17\cdot 19\cdot 21$. Hence $q=6$, $m=11$ and $m-q=5$. 
By Corollary \ref{prodimp} we have $N= 9!! \sum_{j=0}^5 \tf{2^j}{j! (2j-1)!!} \, (6)_j (5)_j$.
Thus,
\begin{multline*} 
N= 9!! \big(1+2\cdot 6\cdot 5 + \tf{4}{2\cdot 3!!} (6\cdot 5) (5\cdot 4) + \tf{8}{3!5!!} 
(6\cdot 5\cdot 4) (5\cdot 4\cdot 3) 
\\ + \tf{16}{4!7!!} (6\cdot 5\cdot 4\cdot 3) (5\cdot 4\cdot 3\cdot 2) + \tf{32}{5!9!!} 
(6\cdot 5\cdot 4\cdot 3\cdot 2) (5\cdot 4\cdot 3\cdot 2\cdot 1)  \big)
\end{multline*}
and after some easy calculations we get
$$N= 9\cdot 7\cdot 5\cdot 3 \cdot (1+60+400+640) + 9\cdot 5\cdot 3 \cdot 1920 + 32\cdot 720 = 
1{.}322{.}685.$$	
Thus, we can also compute $M=12\cdot 14\cdot 16\cdot 18 \cdot 20$ by doing $M=\f{21!}{11!N} = 967{.}680$.
\end{ejem}

\subsection{Congruences mod powers of 2}
Using the previous results we will obtain the values of $\tbinom{2^r m}{2^rq}$ and $\tbinom{2^r m}{2^r q+1}$ modulo $2^t$, 
for any $r$ and small values of $t$.

\sk
We have the equalities 
\begin{equation} \label{bino2} 
\begin{split}
\tbinom{2m}{2q}   & = \tbinom mq \tf{(2m-1)!!}{(2q-1)!!(2(m-q)-1)!!} \\ 
\tbinom{2m}{2q+1} & = 2(m-q)\tbinom mq \tf{(2m-1)!!}{(2q+1)!!(2(m-q)-1)!!}. 
\end{split}
\end{equation}
Since all the double factorials above are odd, we deduce that 
\begin{equation} \label{bino1} 
\tbinom{2m}{2q}\equiv \tbinom mq \pmod 2  
\qquad \quad \text{and} \qquad \quad 
\tbinom{2m}{2q+1}\equiv 0 \pmod 2.
\end{equation} 
Actually, it is well known that $\tbinom{n}{k} \equiv 0$ mod 2 for $n$ even and $k$ odd, and 
$\tbinom{n}{k} \equiv \tbinom{\lfloor n/2 \rfloor}{\lfloor k/2 \rfloor}$ mod 2 otherwise 
(hence by taking $n=2m$ and $k=2q, 2q+1$ one recovers \eqref{bino1}).

Obviously, we have $\tbinom{2^rm}{2^sq+1}\equiv 0$ mod $2$ for every $s\ge 0$.
However, by iterating \eqref{bino1}, we deduce that for any $r\in \N$ we have 
\begin{equation} \label{bino1b} 
\tbinom{2^r m}{2^rq}\equiv \tbinom mq \pmod 2.  
\end{equation} 

As a corollary to Theorem \ref{pochs} we have the following result improving \eqref{bino1b}. 
\begin{prop} \label{prop congs}
Let $m,q \in \N_0$. 

\sk \textit{(a)} For any $r\ge 1$ we have
$$\tbinom{2^r m}{2^r q} \equiv \begin{cases} 
\tbinom mq  & \quad \pmod{2, 4}, \sk \\
\tbinom mq (1+2q(m-q)) & \quad \pmod{8}, \!\!\!\!\! \pmod{16} \text{ with } r=1, \sk \\
\tbinom mq (1+10q(m-q)) & \quad \pmod{16}, r\ge 2.
\end{cases}$$

\sk \textit{(b)} For $1\le r \le 3$ we have
$$\tbinom{2^r m}{2^r q+1}  \equiv 
\begin{cases} 
0   & \quad \pmod{2^r}, \sk \\
2^r(m-q) \tbinom mq  & \quad \pmod{2^{r+1}}. 
\end{cases}$$
\end{prop}

\begin{proof}
(a) By expanding the expressions in \eqref{comb1} for $m,q\in \N_0$ we have
\begin{align*}
\tbinom{2m}{2q} & = \tbinom{m}{q} + 4 \tbinom{m-2}{q-1} \tbinom{m}{2} + 16 \tbinom{m-4}{q-2} \tbinom{m}{4} + 
\text{terms divisible by $64$}, \\ 
%\intertext{and} 
\tbinom{2m}{2q+1} & = 2 m\tbinom{m-1}{q}+ 8 \tbinom{m-3}{q-1} \tbinom{m}{3} + \text{terms divisible by $32$}.
\end{align*}
From this, since $4 \tbinom{m-2}{q-1} \tbinom{m}{2} = 2q(m-q) \tbinom mq$ and $2m \tbinom{m-1}{q} = 2(m-q) \tbinom mq$, it follows that 
\begin{alignat}{2} \label{bino4} 
\begin{split}
\tbinom{2m}{2q} & \equiv \begin{cases} 
\tbinom mq  & \qquad \pmod{ 2, 4}, \sk \\
\tbinom mq (1+2q(m-q)) & \qquad \pmod{ 8, 16},
\end{cases} \\ 
\tbinom{2m}{2q+1} & \equiv 
\begin{cases} 
0   & \qquad \qquad \quad \pmod{ 2}, \\
2(m-q) \tbinom mq  & \qquad \qquad \quad \pmod{4, 8},
\end{cases}
\end{split}
\end{alignat} 
which improves \eqref{bino1}.

The congruences for $\tbinom{2^r m}{2^r q}$ and $\tbinom{2^r m}{2^r q+1}$ in the statement will follow directly from induction on $r$, 
\eqref{bino4} being the inicial step. 
It is clear that $\tbinom{2^rm}{2^rq}\equiv \tbinom mq$ mod $4$ and $\tbinom{2^rm}{2^rq+1} \equiv 0$ mod $2$, for any $r$. 
By using \eqref{bino4} twice we have 
$$\tbinom{4m}{4q} \equiv \tbinom{2m}{2q}(1+8q(m-q)) \equiv \tbinom{2m}{2q} \equiv \tbinom{m}{q}(1+2q(m-q)) \pmod 8,$$
and hence, by induction, for any $r\ge 2$ we have 
$$\tbinom{2^rm}{2^rq} \equiv \tbinom{2^{r-1}m}{2^{r-1}q}(1+2^{2r-1}q(m-q)) \equiv \tbinom{2^{r-1}m}{2^{r-1}q} 
\equiv \tbinom{m}{q}(1+2q(m-q)) \pmod 8.$$

For modulo 16 we proceed similarly, but now 
$$\tbinom{4m}{4q} \equiv \tbinom{2m}{2q}(1+8q(m-q)) \equiv \tbinom{m}{q}(1+2q(m-q))(1+8q(m-q)) \pmod{16}$$ 
and hence
$\tbinom{4m}{4q} \equiv \tbinom{m}{q}(1+10q(m-q))$ mod $16$.
Thus, for any $r\ge 3$ we have
$$\tbinom{2^rm}{2^rq} \equiv \tbinom{2^{r-1}m}{2^{r-1}q}(1+2^{2r-1}q(m-q)) \equiv \tbinom{2^{r-1}m}{2^{r-1}q} \equiv 
\tbinom{m}{q}(1+10q(m-q)) \pmod{16}.$$

\sk (b) By \eqref{bino4}, we have 
$\tbinom{4m}{4q+1} \equiv 4(m-q) \tbinom{2m}{2q}$ mod $4, 8$, and hence
$\tbinom{4m}{4q+1} \equiv  0 \text{ mod } 4$ and 
$\tbinom{4m}{4q+1} \equiv 4(m-q)(1+2q(m-q))\tbinom mq \equiv 4(m-q) \tbinom mq$ mod $8$.
Similarly, one proves that $\tbinom{8m}{8q+1} \equiv 0$ mod $8$ and 
$\tbinom{8m}{8q+1} \equiv 8(m-q)\tbinom mq$ mod $16$, 
and we are done.
\end{proof}

\begin{ejem}
Consider $\tbinom{48}{16} = 2{.}254{.}848{.}913{.}647$ and $\tbinom{56}{17} = 97{.}997{.}533{.}741{.}800$. 
Since $\tbinom{48}{16} = \tbinom{2^4 \cdot 3}{2^4 \cdot 1}$, by Proposition \ref{prop congs}, with $m=3$ and $q=1$, we have 
$\tbinom{48}{16} \equiv \tbinom 31 =3$ mod $4$, $\tbinom{48}{16} \equiv 3 (1+4) \equiv 7$ mod $8$ 
and $\tbinom{48}{16} \equiv 3(1+20) \equiv 15$ mod $16$.  
Similarly, since $\tbinom{56}{17} = \tbinom{2^3\cdot 7}{2^3 \cdot 2}$, we have that $\tbinom{56}{17} \equiv 0$ mod $8$ and 
$\tbinom{56}{17} \equiv 40 \tbinom 72 \equiv 8$ mod $16$.
\end{ejem}

\begin{rem}
Following the same procedure that lead us to \eqref{bino4}, congruences with bigger moduli can be obtained provided we impose some extra  
conditions on $q$ or $m$. For instance, 
$$\tbinom{2m}{2q} \equiv \tbinom mq \{1+2q(m-q) + \tf 23 q(q-1)(m-q)(m-q-1)\} \pmod{32, 64}$$
if $q\equiv 0,1$ mod $3$ or $m-q \equiv 0,1$ mod $3$, and  
$$\tbinom{2m}{2q+1} \equiv 2(m-q) \tbinom mq \{1+\tf 23 q(m-q-1)\}  \pmod{16, 32}$$
if $q$ or $m-q-1$ are divisible by 3.
\end{rem}

\sk
We will next need the arithmetic function $\varepsilon: \N_0 \rightarrow \N_0$, where 
$\varepsilon(k)$ is the biggest power of $2$ dividing $k!$, that is 
\begin{equation} \label{ek}
k!=2^{\varepsilon(k)} \ell_k, \quad \ell_k \text{ odd.}
\end{equation} 
That is, $\varepsilon(k)=\nu_2(k!)$ the $2$-adic valuation of $k!$.
By de Polignac's formula we have 
$\varepsilon(k) = \lfloor \tf{k}{2} \rfloor + \lfloor \tf{k}{2^2} \rfloor + \cdots + \lfloor \tf{k}{2^t} \rfloor$,  
with $t=\lfloor \log_2(k) \rfloor$.
It is thus clear that 
\begin{equation} \label{e2t}
\varepsilon(2^t) = 2^{t-1}+2^{t-2}+\cdots+2+1 = 2^t-1.
\end{equation}
In general, we have the following.
\begin{lema} \label{lemin}
For any $k,r,m \in \N_0$, $m$ odd, we have  
\begin{enumerate}[(a)]
\item $\varepsilon(k) \le k-1$, with equality if $k$ is a power of $2$;

\sk \item $\varepsilon(2^rm) = \varepsilon(m) + (2^r-1)m$; 

\sk \item $\varepsilon(2^r+1) = \varepsilon(2^r-1) +r$; and 

\sk \item $\varepsilon$ is an increasing function (monotonic when restricted to even or to odd numbers). 
\end{enumerate}
\end{lema}

\begin{proof}
First note that for any $k \ge 0$ we have 
\begin{equation} \label{relations}
\varepsilon(2k)=\varepsilon(2k+1) \qquad \text{and} \qquad \varepsilon(2k)=\varepsilon(k)+k.
\end{equation}
The first relation is obvious and the second one follows from 
$(2k)!=2^{\varepsilon(2k)} \, \ell_{2k}$ and $(2k)!=2^k \, k! (2k-1)!! = 2^{\varepsilon(k)+k}\, \ell_k \, (2k-1)!!$ 
and the fact that $\ell_k$, $\ell_{2k}$ and $(2k-1)!!$ are all odd numbers.

Now, the inequality in (a) follows directly by applying strong induction, since
$$\varepsilon(2(k+1)) = \varepsilon(k+1)+k+1 \le 2k+1=2(k+1)-1,$$
where we have used \eqref{relations}. The remaining assertion is clear.

The expression in (b) is obtained by repeated application of the second equality in \eqref{relations}.
Since $(2^r+1)!=(2^r+1)2^r (2^r-1)!$, the expression in (c) is straightforward from the definition of $\varepsilon$. 
Finally, (d) follows from \eqref{relations} and the fact that by definition we have $\varepsilon(k+2)\ge \varepsilon(k)+1$. 
\end{proof}

The following result, which is probably known,  
complements the previous proposition. 
We include a proof for completeness. 
We will need the following notation
\begin{equation} \label{emq}
\varepsilon(m,q) := \varepsilon(m)-\varepsilon(q)-\varepsilon(m-q), \qquad q\le m.
\end{equation} 

\begin{prop} \label{mejora}
For every $m,q,r \in \N$ we have $\varepsilon(m,q)\ge 0$ and
\begin{equation} \label{eq mejora}
\begin{aligned} 
\tbinom{2^r m}{2^r q} \equiv 
   \begin{cases}
      0          & \pmod{2^{\varepsilon(m,q)}}, \\
      \tbinom mq & \pmod{2^{\varepsilon(m,q)+1}}, \end{cases} 
\sk \\ 
\tbinom{2^r m}{2^r q+1} \equiv 
 \begin{cases}
   \tbinom mq & \pmod{2^{\varepsilon(m,q)}}, \\
    0         & \pmod{2^{\varepsilon(m,q)+r}}   . \end{cases} 
\end{aligned}
\end{equation}
In particular, $\tbinom{2^r m}{2^r q} \equiv \tbinom{2^r m}{2^r q+1}$ mod $2^{\varepsilon(m,q)}$.
\end{prop}

\begin{proof}
First note that 
$$\tbinom{m}{q} = \tfrac{2^{\varepsilon(m)}}{2^{\varepsilon(q)} \, 2^{\varepsilon(m-q)}} \tfrac{\ell_m}{\ell_q \,\ell_{m-q}} = 2^{\varepsilon(m)-\varepsilon(q)-\varepsilon(m-q)} \in \Z.$$ 
Thus, since $\ell_m, \ell_k$ and $\ell_{m-q}$ are odd numbers, we have that 
$\tbinom{m}{q} = 2^{\varepsilon(m,q)} \, \ell$,
with $\ell$ odd, and hence $\varepsilon(m,q)\ge 0$.

Now, by (b) of Lemma \ref{lemin} we have 
$$\tbinom{2^r m}{2^r q} = \tf{2^{\varepsilon(2^rm)}}{2^{\varepsilon(2^rq)} \, 2^{\varepsilon(2^r(m-q))}}  
\tf{\ell_{2^r m}}{\ell_{2^r q} \, \ell_{2^r(m-q)}} = \tf{2^{\varepsilon(m)}}{2^{\varepsilon(q) + \varepsilon(m-q)}} \, \ell' $$
with $\ell'$ odd. In this way, we get $\tbinom{2^r m}{2^r q} \equiv 0$ mod $2^{\varepsilon(m,q)}$ and  
$$\tbinom{2^r m}{2^r q} - \tbinom{m}{q} = 2^{\varepsilon(m,q)} (\ell'-\ell),$$
with $\ell, \ell'$ odd, and thus the first congruence is established.

On the other hand, we have
$$\tbinom{2^r m}{2^r q+1} = \tf{(2^rm)!}{(2^rq)!(2^r(m-q))!} \tf{2^r(m-q)}{2^rq+1} = 2^{\varepsilon(m,q)+r} \,  \tf{(m-q) \ell''}{2^rq+1}$$
for some odd integer $\ell''$. Also, $\tbinom{2^r m}{2^r q+1}-\tbinom{m}{q} \equiv 0$ mod $2^{\varepsilon(m,q)+r}$, and the second congruence in the statement follows.

The remaining assertion follows directly from \eqref{eq mejora}.
\end{proof}

In particular, Proposition \ref{mejora} implies that for any $m,q,r \in \N$ we have
\begin{equation} \label{binoweird}
\tbinom{2^rm}{2^rq+s} \equiv \tbinom mq (1-\delta_{s,t}) \pmod{2^{\varepsilon(m,q) +t}}
\end{equation}
where $s,t\in \{0,1\}$ and $\delta_{s,t}$ is the Kronecker $\delta$-function.

\sk 
In certain cases, Lemma \ref{mejora} improves Proposition~\ref{prop congs}. 
This will be the case, for instance, when $\varepsilon(m,q)\ge 4$. 

\begin{prop} \label{prop cong 2t}
Let $r$ and $t$ be natural numbers and for fixed $t$ put $m_t=2^t$, $q_t=2^{t-1}-1$.
\begin{enumerate}[(a)]
\item For $(m,q)=(m_t+1,q_t), (m_t+1,q_t-1)$ or $(m_t,q_t-1)$ we have
$$\tbinom{2^r m}{2^r q} \equiv 0 \pmod {2^{t-1}} \qquad \text{and}  \qquad \tbinom{2^r m}{2^r q} \equiv \tbinom{m}{q} \pmod {2^{t}}.$$ 

\item Moreover, 
$$\tbinom{2^r m_t}{2^r q_t} \equiv 0 \pmod{2^{t}} \qquad \text{and} \qquad 
\tbinom{2^r m_t}{2^r q_t} \equiv \tbinom{m_t}{q_t} \pmod{2^{t+1}}.$$ 
\end{enumerate}
\end{prop}

\begin{proof}
(a) If $t=1$ the result is trivial for mod $2^{t-1}$ and holds by Proposition \ref{prop congs} for mod $2^t$. 
For $t\ge 2$ fixed, consider the numbers $m=m_t+1=2^t+1$ and $q=q_t=2^{t-1}-1$; hence $m-q=2^{t-1}+2$. 
By using \eqref{e2t} and \eqref{relations} we have 
$\varepsilon(m)=\varepsilon(2^t+1) = \varepsilon(2^t)=2^t-1$, 
\begin{align*}
\varepsilon(q) & = \varepsilon(2^{t-1}-2) = \varepsilon\big(2(2^{t-2}-1)\big)= \varepsilon(2^{t-2}-1)+2^{t-2}-1, \\
\varepsilon(m-q) & = \varepsilon(2^{t-1}+2) = \varepsilon\big(2(2^{t-2}+1)\big) = \varepsilon(2^{t-2}+1)+2^{t-2}+1.
\end{align*}
Now, by Lemma \ref{lemin} (c) we have 
$\varepsilon(q) = \varepsilon(2^{t-2}+1) - (t-2)+ 2^{t-2}+1 $. In this way, by \eqref{emq}, we have
\begin{eqnarray*}
\varepsilon(m,q) &=& 2^t-1 -\{2\varepsilon(2^{t-2}+1) +2^{t-1} - (t-2) \} \\
&=& 2^t-1 -\{2(2^{t-2}-1) +2^{t-1} - (t-2)\} \\
&=& 2^t-1 - (2^{t-1}+2^{t-1} -t) = t-1. 
\end{eqnarray*}
The result now follows directly by \eqref{eq mejora} in this case. 
Finally, by using Lemma \ref{lemin}, it is easy to check that  
$$\varepsilon(m_t+1,q_t)=\varepsilon(m_t+1,q_t-1)=\varepsilon(m_t,q_t),$$ 
and hence the statement in (a) follows. 

\sk 
(b) Proceeding similarly as above 
we have $\varepsilon (m_t)=2^t-1$, 
$\varepsilon (q_t)= 2^{t-1}-1 -(t-1)$ 
and $\varepsilon(m_t-q_t)=2^{t-1}-1$. 
Thus, $\varepsilon(m_t,q_t)=t$, 
as we wanted to see.
\end{proof}

\section{Consequences for central binomial coefficients}
We will apply the formulas for BKP's of the previous sections to obtain some new explicit and recursive expressions for the numbers
$$c_m = \tbinom{2m}{m}, \qquad m\ge 0,$$ 
known as \textit{central binomial coefficients}. 
For $0\le m\le 12$ we have 
$1, 2, 6, 20, 70, 252$, $924$, $3432$, $12870$, $48620$, $184756$, $705432$ and $2704156$ (see A000984 in \cite{OEIS}).

By taking $r=1$ and $p=m$ in \eqref{comb4} we get the expression 
\begin{equation} \label{centralsum}
c_m = \sum_{\substack{0\le \ell \le m \\ \ell \equiv m \, (2)}} 
2^\ell \tbinom{m-\ell}{\frac{m-\ell}{2}} \tbinom{m}{\ell} = 
m! \sum_{\substack{0\le \ell \le m \\ \ell \equiv m \, (2)}} 
\tfrac{2^\ell}{\ell! \{(\frac{m-\ell}{2})!\}^2},
\end{equation}
or, distinguishing the cases $m$ even or odd,
\begin{equation} \label{centralsumcases}
\begin{split}
c_{2q}   &= (2q)! \sum_{j=0}^q \tfrac{2^j}{j!(2j-1)!!\{(q-j)!\}^2},              \\
c_{2q+1} &= 2(2q+1)! \sum_{j=0}^q \tfrac{2^j}{j!(2j+1)!!\{(q-j)!\}^2} .
\end{split}
\end{equation}

Also, by taking $m=2q, 2q+1$ in Lemma \ref{prop1} we get 
\begin{equation} \label{centralitos}
\begin{split}
c_{2q}   & = \sum_{j=0}^q 4^j  \tbinom{2q}{2j} \, c_{q-j},              \\
c_{2q+1} & = 2\sum_{j=0}^q 4^j \tbinom{2q+1}{2j+1} \, c_{q-j}.
\end{split}
\end{equation}
The above identities recursively express $c_{2q}$ and $c_{2q+1}$ 
in terms of the first $q+1$ central binomial coefficients $c_0,c_1,\ldots,c_q$.
In other words, we have
\begin{equation} \label{c2q}
\begin{aligned}
c_{2q} & = c_q + 4 \tbinom{2q}{2} c_{q-1} + 4^2 \tbinom{2q}{4} c_{q-2} + \cdots + 4^{q-1} \tbinom{2q}{2q-2} c_1 + 4^q,
 \\ & \\
c_{2q+1} & = 2 \big\{ \tbinom{2q+1}{1} c_q + 4 \tbinom{2q+1}{3} c_{q-1} + 4^2 \tbinom{2q+1}{5} c_{q-2} + 
\cdots + 4^{q-1} \tbinom{2q+1}{2q-1} c_1 + 4^q \big\},
\end{aligned}
\end{equation}
since $c_0=1$.
It is well known that $\tbinom{2m}m$ are even numbers for any $m\ge 1$. This is trivial from \eqref{c2q}, since $c_1=2$.

Furthermore, by taking $m=2q$ in \eqref{b1} and \eqref{stirling} we get the reduction formulas
\begin{equation} \label{4q2q}
c_{2q} = c_q  \, \sum_{j=0}^q \tf{2^j}{j!(2j-1)!!} \, (q)_j^2 =  
c_q  \, \sum_{j=0}^q \tf{2^j}{j!(2j-1)!!} \, \sum_{k,l=0}^q (-1)^{k+l} s(j,k) s(j,l) q^{k+l}
\end{equation} 
expressing $c_{2q}$ in terms of $c_q$.
This allows to give nice expressions for $\tbinom{4q}{2q} / \tbinom{2q}{q}$ and $\tbinom{4q}{2q} - \tbinom{2q}{q}$, and by iteration for 
$\tbinom{2^{r+1}q}{2^r q} / \tbinom{2^{r}q}{2^{r-1}q}$ and $\tbinom{2^{r+1}q}{2^r q} - \tbinom{2^{r}q}{2^{r-1}q}$, for any $r$.

\msk 
We now give alternative expressions to \eqref{centralitos} for the central binomial coefficients, recursively expressing $c_{2q}$ 
and $c_{2q+1}$ 
in terms of fractions involving $c_0,c_1,\ldots,c_q$.
\begin{prop} \label{prop cm alt}
For any $q\in \N$ we have
\begin{equation} \label{c2q alt}
\begin{split}
c_{2q}   & = \tfrac{4q-1}{2q^2} \sum_{j=1}^q 4^j j \tbinom{2q}{2j} c_{q-j}, \\
c_{2q+1} & = \tfrac{2(4q+1)}{(2q+1)^2} \sum_{j=0}^q 4^j (2j+1) \tbinom{2q+1}{2j+1} c_{q-j}.
\end{split}
\end{equation}
\end{prop}
\begin{note}
Notice that $c_{2q}$ depends only on $c_0,c_1,\ldots,c_{q-1}$; compare with \eqref{centralitos}.
\end{note}
\begin{proof}
The result will follow directly by considering $p=m$ in Theorem \ref{teo1} and evaluating the resulting expressions \eqref{Kraw1} at 
$j=0$ and $j=1$. In fact, by \eqref{k0k1} we have $K_\ell^m(0) = \tbinom{m}{\ell}$ and 
$K_\ell^m(1) = (1-\tfrac{2\ell}{m}) \tbinom{m}{\ell}$.
Taking $j=0$ and $j=1$ in \eqref{Kraw1} respectively we have
\begin{equation} \label{j=1}
\tbinom{2m}{m} = \sum_{\substack{ 0\le \ell \le m \sk \\ \ell \equiv m \,(2)}} 
2^\ell \, \tbinom{m-\ell}{\tfrac{m-\ell}2} \tbinom{m}{\ell}
\qquad \text{and} \qquad
K_m^{2m}(2) = \sum_{\substack{ 0\le \ell \le m \sk \\ \ell \equiv m \,(2)}} 
2^\ell \, \tbinom{m-\ell}{\tfrac{m-\ell}2} K_\ell^m(1).
\end{equation}

We now compute $K_m^{2m}(2)$.
By \eqref{eq. kraws}, after some computations, we have 
$$K_p^{n}(2) = \tbinom{n-2}{p} - 2\tbinom{n-2}{p-1} + \tbinom{n-2}{p-2} 
             = \tbinom np \tfrac{(n-p)(n-p-1)-2p(n-p)+p(p-1)}{n(n-1)} $$ 
and hence, taking $p=m$ and $n=2m$ we get 
$$K_m^{2m}(2) = \tbinom{2m}{m} \tfrac{m(m-1)-2m^2+m(m-1)}{2m(2m-1)} = \tfrac{1}{1-2m} \tbinom{2m}{m} .$$ 
Putting all these things together into the second equality in \eqref{j=1} we have
$$\tfrac{1}{1-2m} \tbinom{2m}{m} =  \sum_{\substack{ 0\le \ell \le m \sk \\ \ell \equiv m \,(2)}} 
2^\ell \, \tbinom{m-\ell}{\tfrac{m-\ell}2} (1-\tfrac{2\ell}{m}) \tbinom{m}{\ell} 
= \tbinom{2m}{m} - \tfrac{2}{m} \sum_{\substack{ 0\le \ell \le m \sk \\ \ell \equiv m \,(2)}} 
2^\ell \, \ell  \, \tbinom{m-\ell}{\tfrac{m-\ell}2} \tbinom{m}{\ell},$$
where we have used the first equation in \eqref{j=1}. From this we get
$$ \tbinom{2m}{m} =  \tfrac{2m-1}{m^2} 
\sum_{\substack{ 0\le \ell \le m \sk \\ \ell \equiv m \,(2)}} 
2^\ell \, \ell \, \tbinom{m-\ell}{\tfrac{m-\ell}2} \tbinom{m}{\ell} $$
from which, by taking $m=2q$ and $m=2q+1$, the expressions in \eqref{c2q alt} follows directly after some trivial computations.
\end{proof}

By combining the previous expressions obtained for $c_{2q}$ and $c_{2q+1}$  we get two other recursions for $c_q$ in terms of 
all the previous $c_0,\ldots,c_{q-1}$, one in terms of `even' binomials $\tbinom{2q}{2j}$ and the other in terms of `odd' binomials 
$\tbinom{2q+1}{2j+1}$, for $0\le j\le q$.
\begin{coro}
For every $q\in \N$ we have 
\begin{equation} \label{cq recs}
c_{q} = \sum_{j=1}^q 4^j  \, \tbinom{2q}{2j} \, \big\{\tfrac{4q-1}{2q^2+1}j -1 \big\} \, c_{q-j} = 
\sum_{j=1}^q 4^j \, \tbinom{2q+1}{2j+1} \, \big\{ \tfrac{4q+1}{2q^2}j -1 \big\} \, c_{q-j}.
\end{equation}
\end{coro}

\begin{proof}
By equating the expressions for $c_{2q}$ in \eqref{centralitos} and \eqref{c2q alt} and isolating the contribution for $j=0$ in the sums we get the first equality in \eqref{cq recs}. Proceeding similarly with $c_{2q+1}$, after some more calculations, we get the second equality in \eqref{cq recs}. 
\end{proof}

\begin{ejem}
We now compute $c_8=\tbinom{16}{8}$ by using \eqref{centralitos} and \eqref{c2q alt} and the values
$c_0=1, c_1=2, c_2=6, c_3=20$, $c_4=70$.
Taking $q=4$, by \eqref{centralitos} and \eqref{c2q alt} we respectively have 
\begin{align*} 
& c_8 = \sum_{j=0}^4 4^j \tbinom{8}{2j} c_{4-j} = c_4 + 4\tbinom 82 c_3 + 4^2 \tbinom 84 c_2 + 4^3 \tbinom 86 c_1 + 4^4, 
\sk \\
& c_8 = \tfrac{15}{32}  \sum_{j=1}^4 4^j \tbinom{8}{2j} \, j \, c_{4-j} = 
\tfrac{15}{32} \{ 4\tbinom 82 c_3 + 4^2 \tbinom 84 2c_2 + 4^3 \tbinom 86 3c_1 + 4^5\}, 
\end{align*}
hence $c_8 = 12{.}870 = \tfrac{15 \cdot 27456}{32}$, as one can easily check.
\end{ejem}

\subsubsection*{Central binomial coefficients and Krawtchouk polynomials}
We will now give a mixed relation between central binomial coefficients and BKP's of the form $K_{2t}^{2q}(q)$. For $q$ even, we will get a cancellation rule; while, if $q$ is odd, we will get a recursive formula for $c_{2q}$ in terms of  
$c_0,c_1,\ldots,c_{q-1}$. 
We will first need the following result.
\begin{lema} \label{k2ppp}
For any $q\in \N_0$ we have
$$K_{2q}^{4q}(2q)=(-1)^q \tbinom{2q}{q} \qquad \text{and} \qquad K_{2q+1}^{4q+2}(2q+1)=0.$$
\end{lema}
\begin{proof}
By \eqref{eq. kraws} we have
$$K_p^{2p}(p) = \sum_{j=0}^p (-1)^j \tbinom pj \tbinom{p}{p-j} = \sum_{j=0}^p (-1)^j \tbinom pj ^2,$$
and it is well known that 
$$\sum_{j=0}^p (-1)^j \tbinom pj ^2 = 
\begin{cases} 0 & \qquad \text{$p$ odd,} \sk \\ (-1)^{p/2} \binom{p}{p/2} & \qquad \text{$p$ even.} \end{cases}$$
Hence, the result follows directly by considering the cases $p=2q$ and $p=2q+1$. 
\end{proof}

It is known that $K_k^n(\tfrac n2)=0$ for $k$ odd and $K_k^n(\tfrac n2) = (-1)^{k/2}\tbinom{n/2}{k/2}$ for $k$ even 
(see for instance (7) in \cite{KL}).
We have included a direct proof of the case that we need for completeness.

\begin{prop} \label{central}
Let $q$ be natural number. 
\begin{enumerate}[(a)] 
\item If $q$ is even then
\begin{equation} \label{suma central}
\sum_{t=1}^q 4^t \, c_{q-t} \, K_{2t}^{2q}(q) =0.
\end{equation}

\item If $q$ is odd then 
\begin{equation} \label{suma central odd}
c_{2q} = - \sum_{t=1}^q 2^{2t-1} \, c_{q-t} \, K_{2t}^{2q}(q).
\end{equation}
\end{enumerate}
\end{prop} 

\begin{proof}
By applying Theorem \ref{teo1} with $m=2q=p$, $j=q$, we get
$$K_{2q}^{4q}(2q) = \sum_{\substack{ 0\le \ell \le 2q \\ \ell \text{ even}}}  
 2^\ell \, \tbinom{2q-\ell}{\f{2q-\ell}{2}} \, K_\ell^{2q}(q) .$$
Also, by the previous Lemma we have $K_{2q}^{4q}(2q)=(-1)^q \tbinom{2q}{q}$. 
Thus, by equating these values we obtain
$$(-1)^q \tbinom{2q}{q} = \tbinom{2q}{q} + \sum_{\substack{ 2 \le \ell \le 2q \\ \ell \text{ even}}}  
 2^\ell \, \tbinom{2q-\ell}{\f{2q-\ell}{2}} \, K_\ell^{2q}(q),$$
that is to say
$$ \tbinom{2q}{q} \big( (-1)^q-1 \big) = \sum_{t=1}^q  4^t \, \tbinom{2(q-t)}{q-t} \, K_{2t}^{2q}(q).$$
It is clear from this identity that we get the expressions in the statement, taking $q$ even or odd respectively. 
\end{proof}

\begin{ejem}
If $q=3$, the sum in \eqref{suma central odd} equals
$$\tbinom 63= -\big( 2 \tbinom 42 K_2^6(3) + 2^3 \tbinom 21 K_4^6(3) + 2^5 \tbinom 00 K_6^6(3) \big) = -(12(-3) + 16 \cdot 3 - 32) = 20.$$
For $q=4$, the sum in \eqref{suma central} is $4\tbinom 63 K_2^8(4) + 4^2 \tbinom 42 K_4^8(4) + 4^3 \tbinom 21 K_6^8(4) + 4^4 \tbinom 00 K_8^8(4)$
which equals $80(-4) + 96 \cdot 6 + 128(-4) +256=0$, as it should be. 
\end{ejem}

\section{Catalan numbers}
For $n\ge 0$, the Catalan numbers 
$$C_n = \tfrac{(2n)!}{n!(n+1)!} = \tfrac{(2n)!}{(n+1)(n!)^2},$$ 
which appear in several different counting problems,
are closely related with central binomial coefficients because of the relation 
\begin{equation} \label{Cn}
C_n = \tf{1}{n+1} \tbinom{2n}{n} = \tbinom{2n}{n}-\tbinom{2n}{n+1}. 
\end{equation}

The first seventeen Catalan numbers are (see A000108 in \cite{OEIS})
\begin{gather*}
C_0=1, \: C_1=1, \: C_2=2, \: C_3=5, \: C_4=14, \: C_5=42, \: C_6=132, \\ 
C_7=429, \: C_8=1{.}430, \: C_9=4{.}862, \: C_{10}=16{.}796, \: C_{11}=58{.}786,   \: C_{12}=208{.}012, \\ 
C_{13}=742{.}900, \: C_{14}=2{.}674{.}440, \: C_{15}=9{.}694{.}845,  \: C_{16}=35{.}357{.}670. 
\end{gather*}
Note that $C_3$, $C_7$ and $C_{15}$ are odd. It is a classic result that $C_n$ is odd if and only if $n$ is a Mersenne number, i.e.\@ $n=M_a=2^a-1$ for some $a\ge 0$ (see for instance \cite{AK}).

\sk
Note that by \eqref{Cn} we have  $c_n=(n+1)C_n$, hence all the expressions obtained for central binomial coefficients in the previous section give rise to similar expressions for Catalan numbers.
For instance, by \eqref{centralsum} we have
\begin{equation} \label{Cn exp}
C_m = \tfrac{1}{m+1} \sum_{\substack{0\le \ell \le m \\ \ell \equiv m \, (2)}} 
2^\ell \tbinom{m-\ell}{\frac{m-\ell}{2}} \tbinom{m}{\ell} = 
\tfrac{m!}{m+1} \sum_{\substack{0\le \ell \le m \\ \ell \equiv m \, (2)}} 
\tfrac{2^\ell}{\ell! \{(\frac{m-\ell}{2})!\}^2}.
\end{equation}
By using \eqref{centralsumcases} one gets similar expressions for $C_{2q}$ or $C_{2q+1}$.

\subsubsection*{Recursions} 
It is known that Catalan numbers satisfy the recursions 
$$C_{n+1}= \tfrac{2(2n+1)}{n+2} C_n \qquad \text{ and } \qquad C_{n+1} = \sum_{k=0}^n C_k C_{n-k},$$
for any $n\ge 0$, expressing $C_{n+1}$ in terms of all the previous numbers $C_0,C_1,\ldots,C_n$. 
By using recursions between binomial coefficients already obtained, it is possible to give other recursion formulas for Catalan numbers, in which $C_{2n}$ and $C_{2n+1}$ are linear combinations of $C_0,C_1,\ldots,C_n$ only.

\begin{prop} \label{prop catalan}
For any non negative integer $n$ we have
\begin{equation} \label{catalan}
\begin{split}
C_{2n}   & = \tf{1}{2n+1} \sum_{k=0}^n \, 4^k \, (n-k+1) \, \tbinom{2n}{2k} \, C_{n-k}, \\
C_{2n+1} & = \tf{1}{n+1}  \sum_{k=0}^n \, 4^k \, (n-k+1) \, \tbinom{2n+1}{2k+1} \, C_{n-k}.
\end{split}
\end{equation}
\end{prop}

\begin{proof}
The result follows directly from the relation $C_m = \tf{1}{m+1} \binom{2m}{m}$ by applying \eqref{centralitos} with $m=2n$ and $m=2n+1$, and then using the first equality in \eqref{Cn} again. 
\end{proof}

Also from \eqref{c2q alt}, by using \eqref{Cn}, we get the alternative recursive expressions 
\begin{equation} \label{C2q rec alt}
\begin{split}
C_{2n}   & = \tfrac{4n-1}{(2n+1) 2n^2} \sum_{k=1}^n 4^k k (n-k+1) \tbinom{2n}{2k} C_{n-k}, \\
C_{2n+1} & = \tfrac{4n+1}{(n+1)(2n+1)^2} \sum_{k=0}^n 4^k (2k+1) (n-k+1) \tbinom{2n+1}{2k+1} C_{n-k}.
\end{split}
\end{equation}
Similarly, two more recursive expressions for Catalan numbers can be obtained from \eqref{cq recs} by using \eqref{Cn}.

\sk 
Note that 
\eqref{catalan} and \eqref{C2q rec alt} are very similar to Touchard's identity
\begin{equation}    \label{touchard}
C_{n+1} = \sum_{k=0}^{[n/2]} 2^{n-2k} \, \tbinom{n}{2k} \, C_k  \qquad (n\ge 0) 
\end{equation}
(see for instance \cite{Shap} and the references therein) which also enables one to recursively obtain $C_{2n}$ from $C_0,\ldots,C_{n-1}$ and $C_{2n+1}$ from $C_0,\ldots,C_n$. 
Another Touchard-type identities are 
\begin{equation}    \label{callan}
C_n = \tfrac{n+2}{n(n-1)} \sum_{k=1}^{\lfloor n/2 \rfloor}  2^{n-2k} \, k \, \tbinom{n}{2k} \, C_k, 
\end{equation}
for $n\ge 2$, proved by Callan (\cite{Call}) and the very similar ones 
\begin{equation}    \label{hurtado}
C_{n+1} = (n+3)  \sum_{k=0}^{\lfloor \frac{n-1}{2} \rfloor}  \tfrac{1}{k+2} \, 2^{n-2k} \, \tbinom{n-1}{2k} \, C_k,  
\end{equation}
for $n\ge 0$, due to Hurtado-Noy (\cite{HN}), and
\begin{equation}    \label{amde}
C_n = \tfrac{n+3}{2n}     \sum_{k=0}^{\lfloor \frac{n-1}{2} \rfloor}  \tfrac{2k+1}{k+2} \, 2^{n-2k} \, \tbinom{n}{2k+1} \, C_k, 
\end{equation}
for $n\ge 0$, obtained by Amdeberhan (according to \cite{Call}).

\begin{ejem} 
We will compute $C_8$ in four ways. If $n=4$, by \eqref{catalan} we have  
\begin{eqnarray*}
C_8 &=& \tf 19 \sum_{k=0}^4 4^k (5-k) \tbinom{8}{2k} C_{n-k} \\ &=& 
\tf 19 \big\{ 5 \tbinom 80 
C_4 + 16 \tbinom 82 C_3 + 4^2 3 \tbinom 84 C_2 + 4^3 2\tbinom 86 C_1 + 4^4 \tbinom 88 
C_0 \big\} 
\end{eqnarray*}
 and by \eqref{C2q rec alt} we have
$$C_8 = \tfrac{5}{96} \sum_{k=1}^4  4^{k} k (5-k) \tbinom{8}{2k} C_{4-k} = 
\tfrac{5}{96} \big\{ 16 \tbinom 82 C_3 + 4^2 6 \tbinom 84 C_2 + 4^3 6 \tbinom 86 C_2 + 4^5 \tbinom 88 C_1 \big\}.$$

Alternatively, using \eqref{touchard} we have $n=7$ and
$$C_8 = \sum_{k=0}^3 2^{7-2k} \tbinom{7}{2k} C_k = 
2^7 \tbinom 70 C_0 + 2^5 \tbinom 72 C_1 + 2^3 \tbinom 74 C_2 + 2^1 \tbinom 76 C_3,$$ 
and also, using \eqref{callan}, we get 
$$C_8 = \tfrac{5}{28} \sum_{k=1}^4  2^{8-2k} k \tbinom{8}{2k} C_k = 
\tfrac{5}{28} \big\{ 2^6 \tbinom 82 C_1 + 2^4 \tbinom 84 2C_2 + 2^2 \tbinom 86 3 C_3 + 2^0 \tbinom 88 4 C_4 \big\}.$$

It is reassuring that, since $C_0=C_1=1, C_2=2, C_3=5$ and $C_4=14$, in all the cases we get the value $C_8=1430$. 
One can also use expressions \eqref{hurtado} and \eqref{amde}.
\end{ejem}

\begin{rem}
By \eqref{Cn}, Proposition \ref{central} relates Catalan numbers with BKP's. In fact, for $q$ even we have the 
`orthogonality' relation 
\begin{equation} \label{Cq ortog}
\sum_{t=1}^q 4^t \, (q-t+1) \, C_{q-t} \, K_{2t}^{2q}(q) =0 \qquad \text{($q$ even)},
\end{equation}
and for $q$ odd we have the recursion
\begin{equation} \label{Cq Kt}
C_q = -\tf{1}{q+1} \sum_{t=1}^q 2^{2t-1} \, (q-t+1) \, C_{q-t}  \, K_{2t}^{2q}(q) \qquad \text{($q$ odd)},
\end{equation}
both involving integral values of BKP's of the form $K_2^{2q}(q), K_4^{2q}(q), \ldots, K_{2q}^{2q}(q)$.
\end{rem}

\subsubsection*{Congruences modulo 2, 4, 8 and 16}
Notice that the the recursions \eqref{catalan} --  \eqref{callan} seem well suited to study congruences of Catalan numbers modulo powers of 2 (this is not the case for \eqref{hurtado} and \eqref{amde} since they involve fractions). In fact, by expanding these expressions and reducing modulo $2^r$, for some $1\le r\le n$, one can obtain congruence relations for 
$C_{2n}$ and $C_{2n+1}$ in terms of $C_n$ and $C_{n-1}$ mod~$2^r$. 

The simplest expressions are the ones obtained from Touchard's identity.
By considering the cases $n$ even or odd separately in \eqref{touchard}, we get
\begin{equation*} \label{touchard2}
C_{2n} = \tf 12 \sum_{k=0}^{n-1} 4^{n-k} \, \tbinom{2n-1}{2k} \, C_k,  
\qquad \text{and} \qquad  
C_{2n+1}  = \sum_{k=0}^{n} 4^{n-k} \, \tbinom{2n}{2k} \, C_k, 
\end{equation*}
and by expanding these expressions one can easily deduce that 
\begin{alignat}{4} \label{touch cong}
\nonumber & C_{2n} \equiv 0               && \pmod 2, && \quad \qquad  C_{2n+1}  \equiv C_n      		  && \pmod 2,  \\
          & C_{2n} \equiv 2C_{n-1}        && \pmod 4, && \quad \qquad  C_{2n+1}  \equiv C_n           && \pmod 4,  \\
\nonumber & C_{2n} \equiv 2(2n-1)C_{n-1}  && \pmod 8, && \quad \qquad  C_{2n+1}  \equiv C_n-4nC_{n-1} && \pmod 8,
\end{alignat}
and 
\begin{equation}
\begin{aligned} \label{touch cong2}
C_{2n}   & \equiv  2(2n-1)C_{n-1}  + 8\tbinom{2n-1}{3} C_{n-2}  & \pmod{16},  \\
C_{2n+1} & \equiv  C_n + 4n(2n-1)C_{n-1} 											  & \pmod{16}.
\end{aligned}
\end{equation}

From the congruences mod 2 above, considering $n=2m$ and $n=2m+1$, we get 
$$C_{4m+1}\equiv C_{2m} \equiv 0 \pmod 2 \quad \text{and} \quad 
C_{4m+3} \equiv C_{2m+1} \equiv C_m \pmod 2.$$ 
Taking $m=2k$ and $m=2k+1$ above we get
\begin{gather*}
C_{8k+5}\equiv C_{8k+3}\equiv C_{8k+1} \equiv C_{4k+2} \equiv C_{4k+1} \equiv C_{4k} \equiv C_{2k} \equiv 0 \pmod 2, \\ 
C_{8k+7} \equiv C_{4k+3} \equiv C_{2k+1} \equiv C_k \pmod 2.
\end{gather*}
Iterating this process, for every $k,\ell \ge 1$ one has that 
\begin{equation} \label{cong Ck j}
C_{2^k \ell + j} \equiv_{_2}
\begin{cases} 
0       & \qquad \text{if } 1 \le j < 2^k-1, \sk \\
C_\ell  & \qquad \text{if } j=2^k-1, 
\end{cases}
\end{equation}
where $\equiv_{_2}$ denotes congruence modulo 2.
In particular, if $\ell=1$ in \eqref{cong Ck j}, for every $k\ge 1$, taking $j=2^k-1$ we get 
$$C_{2^{k+1}-1} \equiv C_1 \equiv 1 \pmod 2.$$ 
From this and \eqref{cong Ck j} we recover the fact that $C_n$ is odd if and only if $n$ is a Mersenne number.

\sk 
Now, considering the cases even and odd separately in Callan's identity \eqref{callan} we obtain
\begin{equation*} \label{callan2}
%\begin{aligned}
C_{2n} = \tf{n+1}{n(2n-1)} \sum_{k=1}^{n} 4^{n-k} \, k\, \tbinom{2n}{2k} \, C_k 
\quad \text{and} \quad  
C_{2n+1}  = \tf{2n+3}{n(2n+1)} \sum_{k=1}^{n} 4^{n-k} \, k\, \tbinom{2n+1}{2k} \, C_k; 
\end{equation*}
and by expanding these expressions one gets 
\begin{alignat}{3}
\nonumber
nC_{2n} & \equiv 0	& \pmod 2, \\ 
n(2n-1)C_{2n} & \equiv n(n+1)C_n	& \pmod{ 4, 8}, \\ \nonumber
n(2n-1)C_{2n} & \equiv (n+1)\{ 4n(n-1)(2n-1) C_{n-1} +nC_n \}	& \pmod{16}, 
\end{alignat}
\begin{alignat}{3}
\nonumber
nC_{2n+1} & \equiv nC_n  																			& \pmod 2, \\
n(2n+1)C_{2n+1} & \equiv 3nC_n 																& \pmod 4, \\ \nonumber
n(2n+1)C_{2n+1} & \equiv 4(n-1)\tbinom{2n+1}{3} C_{n-1} - (4n^2+3)n C_n & \pmod{8, 16},
\end{alignat}
or equivalently $nC_{2n+1} \equiv n(2n+3)C_n$ mod 4.

\sk
On the other hand, by expanding \eqref{catalan} in Proposition \ref{prop catalan}, it follows directly that 
\begin{alignat}{3} \label{6.14}
\nonumber
C_{2n} & \equiv (n+1) C_n 															& \pmod 2, \\
(2n+1) C_{2n} & \equiv (n+1) C_n 												& \pmod 4, \\ \nonumber
(2n+1) C_{2n} & \equiv (n+1) C_n - 4n^2 C_{n-1}  & \pmod{8}, \\ \nonumber
(2n+1) C_{2n} & \equiv (n+1) C_n + 4n^2 (2n-1) C_{n-1}  & \pmod{16}, 
\end{alignat}
\begin{alignat}{3} \label{6.15}
\nonumber
(n+1) C_{2n+1} & \equiv (n+1) C_n  																			& \pmod 2, \\
(n+1) C_{2n+1} & \equiv (n+1) (2n+1) C_n 																& \pmod 4, \\ \nonumber
(n+1) C_{2n+1} & \equiv (n+1) (2n+1) C_n + 4n \tbinom{2n+1}{3} C_{n-1}  & \pmod{8, 16}. 
\end{alignat}

Congruences modulo $32$ and $64$, or even higher powers of $2$, can also be obtained in the same way, although with a fast increasing complexity.

\begin{rem} 
The determination of the Catalan numbers mod 4 (resp.\@ 8) is given in Theorems 2.3 (resp.\@ 4.2) in \cite{ELY} by using 
\textsl{ad hoc} methods. In the mod 4 case, if we put 
$\mathcal{C}_4(i) = \{C_n : C_n\equiv i \text{ (mod } 4)\}$ for $0\le i\le 3$ 
and $N_{a,b} = 2^a +2^b-1 = 2^a+M_b$ 
then 
$\mathcal{C}_4(0) = \{C_n \,:\, n \ne N_{a,b}, a>b\ge 0 \}$, $\mathcal{C}_4(1) = \{C_n \,:\, n=M_a, a \ge 0 \}$  
$\mathcal{C}_4(2) = \{C_n \,:\, n = N_{a,b}, \, a>b\ge 0 \}$ and $\mathcal{C}_4(3) = \varnothing$.
Similar results hold for the mod 8 case.
By using \eqref{touchard}, shorter and easier proofs of these facts can be found in \cite{XX}, where also a systematic approach to Catalan numbers 
modulo $2^r$ is carried out.
\end{rem}

\section*{final remarks}
(a) The method used to obtain Theorem \ref{teo1} does not seem to apply for elements $x\in T_{2m}$ of order $>2$ because the $p$-traces 
$\chi_p(x)$ are not expressible, a priori, in terms of (binary) Krawtchouk polynomial, as in \eqref{pr1}. 

\sk 
(b) Is there any combinatorial proof for (or explanation to) each of the expressions 
\eqref{Cn exp} -- \eqref{C2q rec alt} obtained for Catalan numbers?

\sk 
(c) The expressions for BKP's obtained so far 
seem well suited to study the values $K_{p}^{2^rm}(2^sj)$ modulo high powers of $2$. 

\sk 
(d) The techniques and results in this paper could be of some utility in studying recursions and congruences for the Motzkin numbers
$M_n$ because of the relations 
$M_n = \sum_{k=0}^\ell \tbinom{n}{2k} C_k$ with $\ell= {\lfloor n/2 \rfloor}$ and $C_{n+1} = \sum_{k=0}^n \tbinom nk M_k$.

\end{document}